\documentclass{amsart}
\usepackage{amssymb}
\usepackage{hyperref}
\usepackage{amsfonts}
\usepackage{amsmath}
\usepackage{amsfonts}
\usepackage[numbers]{natbib}
\usepackage{draftwatermark}
\SetWatermarkText{}
\SetWatermarkScale{5}
\SetWatermarkLightness{.9}

\newtheorem{theorem}{Theorem}

\newtheorem{corollary}[theorem]{Corollary}

\newtheorem{definition}[theorem]{Definition}

\newtheorem{lemma}[theorem]{Lemma}

\newtheorem{proposition}[theorem]{Proposition}
\newtheorem{property}[theorem]{Property}

\newcommand{\sigmahat}{{\sigma}}

\title{self-induced Systems}
\date{\today}
\author{Fabien Durand$^1$, Nicholas Ormes$^2$, Samuel Petite$^1$}
\begin{document}
\bibliographystyle{plain}

\email{normes@du.edu}
\email{fabien.durand@u-picardie.fr}
\email{samuel.petite@u-picardie.fr}

\address{$^1$Laboratoire Ami\'enois
de Math\'ematiques Fondamentales et Appliqu\'ees, CNRS-UMR 7352, Universit\'{e} de Picardie Jules Verne, Amiens, France.} 

\address{$^2$ Department of Mathematics, University of Denver, Denver, United States}

\thanks{The first author would like to thank the ANR program Dyna3S. The second author thanks the Universit\'e de Picardie Jules Verne for the visiting professor position which at the origin of this work.}


\begin{abstract}
A minimal Cantor system is said to be self-induced whenever it is conjugate to one of its induced systems. 
Substitution subshifts and some odometers are classical examples, and we show that these are 
the only examples in the equicontinuous or expansive case.
Nevertheless, we exhibit a zero entropy self-induced system that is
neither equicontinuous nor expansive. We also provide non-uniquely
ergodic self-induced systems with infinite entropy.
Moreover, we give a characterization of self-induced minimal Cantor systems in terms of  substitutions on finite or infinite alphabets.
\end{abstract}

\maketitle

\section{Introduction}
From the Poincar\'e's recurrence theorem one can define the notion of a first return map on a set of positive measure and, as a consequence, what is now named an induced dynamical system.
In the 1940's S. Kakutani initiated, in the measurable framework, the study of such systems. 
Intensive studies of the induction structures for fifty years contributed to a better understanding of the orbital structures of dynamical systems, see for example \cite{Ornstein&Rudolph&Weiss:1982,Kosek&Ormes&Rudolph:2008,Roychowdhury&Rudolph:2009,delJunco&Rudolph&Weiss:2009,Dykstra&Rudolph:2010}.
However, in the topological context this notion has been less studied despite the powerful strategy it provides as shown in \cite{Giordano&Putnam&Skau:1995} (see also \cite{Herman&Putnam&Skau:1992}).

\medskip

Since, it has been observed in several well-known families of dynamical systems the phenomenon of self-induction, that is, systems being conjugate to one of its induced systems.

For interval exchange transformations (IETs), where induction plays a crucial role through the so-called Rauzy-Veech induction \cite{Veech:1978,Rauzy:1979}, a stationary behaviour my appear in the scheme of successive inductions.
The IETs whose interval lengths lie in a quadratic number field, as quadratic rotations, are conjugate to an induced system on a subinterval  \cite{Boshernitzan&Carroll:1997}.
Let us point out this is not particular to the quadratic fields since there are IETs defined on cubic fields that are self-induced on canonical intervals \cite{Arnoux&Rauzy:1991,Arnoux&Ito:2001, Ferenczi&Holton&Zamboni:2003,Bressaud&Jullian:2012,Jullian:2013} and even for some algebraic fields of arbitrary degree \cite{Arnoux:1988}.

Let us also mention that there exist self-induced exchanges of domains like for the Rauzy fractal \cite{Rauzy:1982} where the conjugacy is given by a similarity. 
A last example is given by the Pascal adic transformation \cite{Mela&Petersen:2005} that can be proven to be self-induced. 
In most of the former papers the self-induction property comes from a
measure theoretical conjugacy to a substitution subshift.
This observation was enlightened in \cite{Arnoux&Ito:2001} where it is shown that any system with a $\sigma$-structure, that is a sequence of finite partitions having specific properties, is conjugate to a substitution subshift. 

\medskip

The self-induction property is sensitive to the class of sets on which the induction is made.
For instance, inducing on measurable sets instead of intervals, each irrational rotation on the torus is an induced system of any other irrational rotation  \cite{Ornstein&Rudolph&Weiss:1982}, thus all rotations are self-induced (in the measure theoretical framework).
A similar phenomena arises in the topological context.
Indeed, we show in Section \ref{section:poincare} that given two minimal systems $(X,T)$ and $(Y,S)$, on Cantor sets, there always exists a closed set $C\subset X$ such that the induced map $T_C$ is well-defined  and $(C, T_C)$ is topologically conjugate to $(Y,S)$ (Theorem \ref{theo:weakkakutanieq}).
This incidentally proves that all such systems are self-induced on some proper closed subsets.

\medskip

In this paper we focus on  minimal topological dynamical systems $(X,T)$ with a self-induced behaviour: that is topological dynamical systems that are conjugate to 
one of their induced systems on proper open  (and necessarily closed) sets. The minimality of the system ensures the induced systems are well defined. 
Since a self-induced system has  closed and open sets, called clopen sets, it is natural to consider dynamics on Cantor sets, called Cantor systems. 
From the Abramov formula for the entropy of induced  systems, the  self-induced property  implies that the entropy is 0 or $+\infty$ (see Proposition \ref{prop:entropy}).
Among zero entropy systems, it is a folklore result that $p$-odometers ({\em i.e.}, addition of $1$ in the set $\mathbb{Z}_p$ of $p$-adic integers) 
or minimal substitution subshifts  \cite{Queffelec:1987} are self-induced (see propositions \ref{prop:chara_selfinduced_odo} and \ref{prop:SubSelf}). 
 
 The goal of this paper is to characterize self-induced minimal Cantor systems. 
In Section \ref{sec:preliminaries}, we give preliminary results on such systems. Among them, we show that the  clopen set on which the system is self-induced can be taken inside any open set.

The next section is devoted to the characterization of self-induced systems in the  equicontinuous  and expansive cases. 
 For the equicontinuous  case  ({\em i.e.}, odometers) we prove: an odometer is self-induced if and only if  it factorizes onto a $p$-odometer  (Proposition \ref{prop:chara_selfinduced_odo}). Hence,  the odometer built as the inverse limit of $(\mathbb{Z}/p_1\cdots p_n \mathbb{Z} )$, where $(p_i)_i$ is the sequence of primes, is not self-induced whereas the 2-odometer (thus, on $\mathbb{Z}_2$) is.
For  expansive minimal Cantor system ({\em i.e.}, minimal subshift) it is equivalent to be conjugate to a minimal substitution subshift  (Proposition \ref{prop:SubSelf} and Theorem \ref{submain}).
 
 We provide in Section \ref{sec:nonEquicnonExp} two examples of self-induced minimal Cantor systems that are neither equicontinuous nor expansive. 
 One example has zero entropy and the other $+\infty$. Moreover this last example is not uniquely ergodic. 

In Section \ref{sec:charSelfInduced}, we characterize self-induced minimal Cantor systems by means of 
  \emph{generalized substitution subshifts} {\em i.e.},   systems 
generated by  substitution maps on a compact zero-dimensional (but not
necessarily finite) alphabet (theorems \ref{gen1} and \ref{gen2}).

In  the last section, we define the notion of Poincar\'e sections and we prove Theorem \ref{theo:weakkakutanieq}.


\section{Preliminaries}\label{sec:preliminaries}


\subsection{Dynamical systems}


For us a \emph{topological dynamical system}, or just a dynamical system, 
is a couple $(X,T)$ where $X$ is a compact metric space $X$ and $T$ is a homeomorphism $ T:X\rightarrow X$. A dynamical system is \emph{minimal} if every orbit is dense in $X$, or equivalently if the only non empty closed invariant set is \( X \). 
We call $(X,T)$ a {\em Cantor system} if $X$ is a Cantor set, 
that is, if \( X \) is a compact metric space with a countable basis of 
closed and open sets and  no isolated points.
We recall that any such space is homeomorphic to the standard Cantor ternary set.
When $X$ is a subset of $A^\mathbb{Z}$ where $A$ is a finite set
and $S$ is the shift map $S(x)_n = x_{n+1}$, the Cantor 
system $(X,S)$ is called a {\em subshift}.
The system  $(X,T)$ is said to be {\em aperiodic} if $T$ has no periodic points. Observe that minimal
Cantor systems are always aperiodic. 

Let $(X_1,T_1)$ and $(X_2,T_2)$ be two dynamical systems.
We say $(X_2,T_2)$ is a {\em factor}\index{factor} of $(X_1,T_1)$ if there is a continuous and onto map $\varphi : X_1 \rightarrow X_2$ such that $\varphi \circ T_1 = T_2\circ \varphi$.
Then, $\varphi $ is called {\em factor map}.
If $\varphi$ is one-to-one we say it is a {\em conjugacy}, and $(X_1,T_1)$ and $(X_2,T_2)$ are {\em conjugate}.

For a minimal system $(X,T)$ and for any open set $U \subset X$, we can define the {\em return time function} $r_U: X \to \mathbb{N}$ by 
$$ r_{U} \colon x  \mapsto \inf\{n > 0  : T^n(x) \in U  \}.$$ 

It is well known that for a minimal Cantor system and $A \subset X$ a clopen set, the 
function $r_A$ is locally constant hence continuous. 
The {\em induced map} $T_A:A \to A$ is then defined by 
$$ T_{A} \colon x \mapsto T^{r_{A}(x)}(x).$$    
Such a map $T_A:A \to A$ is a homeomorphism and the Cantor system $(A, T_A)$ 
is called the {\em induced system} on $A$. 
If the system $(X,T)$ is minimal 
and $A \subset X$ is clopen 
then, $(A,T_A)$ is also minimal. 

Conversely, let $(X,T)$ be a Cantor system  and $r \colon X \to {\mathbb Z}_{+}$ a continuous map, the {\em exduced system} $(\tilde{X}_{r}, \tilde{T})$ 
is the dynamical system defined by the set 
$\tilde{X}_{r} = \{(x, i) : x \in X,  0 \le i < r(x) \}$ and the map 
$\tilde{T} :\tilde{X}_{r} \to \tilde{X}_{r}$ defined by  
$$(x,i)  \mapsto  \begin{cases} 
(x, i+1) & \textrm{ if } 0 \leq  i < r(x)-1 \\ 
(Tx, 0) & \textrm{ if } i = r(x).
\end{cases}
$$
It is plain to check that the exduced system of $(X,T)$ is a minimal Cantor system
if $(X,T)$ is. 
Moreover, the induced system of $(\tilde{X}_{r}, \tilde{T})$ on the clopen set 
$X \times \{0\}$ is conjugate to the system $(X, T)$. 
More precisely, if $(X,T)$ is a minimal Cantor system, $U$ a clopen set of $X$ and 
$r_{U}$ the associated return function then, $(x,i)  \mapsto T^{i}x$
is a conjugacy between the systems $(\tilde{U}_{r_{U}}, \tilde{T}) $ and $(X,T)$. 

\begin{definition}\label{def:SI}
A minimal Cantor system $(X,T)$ is  \emph{self-induced} 
if there exists a non-empty clopen proper subset $U \subset X$ such that 
the induced system $(U, T_U)$ is conjugate to the system $(X,T)$.
\end{definition}

While we work within the topological category in our study of self-induced systems, 
it will be relevant to consider invariant Borel probability measures. By 
the Bogoliouboff-Krylov theorem, associated
to any topological dynamical system $(X,T)$ there is a nonempty compact, convex space $M(X,T)$ of 
$T$-invariant Borel probability measures \cite{KryBog:1937}, see also \cite{Petersen:1983}. 
When $M(X,T)$ is a singleton set, as is the case for minimal substitution subshifts and odometer systems, 
we say that $(X,T)$ is \emph{uniquely ergodic}. The proposition below is well-known and describes the relation
between $M(X,T)$ and the space of measures for an induced system of $(X,T)$. 

\begin{proposition}
\label{prop:inducedmeasure}
Let $(X,T)$ be a minimal Cantor system and suppose $U \subset X$ is clopen. 
Let $F:M(X,T) \to M(U,T_U)$ be the function where 
$F(\mu)$ is the measure $\nu$ on 
$U$ defined by $\nu(B) = \frac{\mu(B)}{\mu(U)}$ where $B \subset U$ is a Borel set. 
Then, $F$ is a bijection. 
\end{proposition}

\begin{proof}
First note that any such measure $\nu = \frac{1}{\mu(U)}\mu$ defines a Borel probability measure on $U$. 
Also, for any Borel set $B \subset U$,  
$$\nu(T_U^{-1}B) = \frac{\mu(T_U^{-1}B)}{\mu(U)} = \frac{\mu(B)}{\mu(U)} = \nu(B).$$
Thus, the map $F$ is well-defined. 

We will prove that $F$ is bijective. To this end, let 
$U_k = \{ x \in U : r_U(x)=k \}$. 
Then, $U_k$ is clopen for all $k$, and for some $K$, $U$ is a disjoint union 
of $U_1, U_2, \ldots , U_K$. 
It follows that the collection $\{ T^i U_k : 0 \leq i < k , 1 \leq k \leq K \}$ forms a clopen partition of $X$. 

Suppose there exist two measures $\mu, \mu '$ in $M (X,T)$ such that $F(\mu)(B) = F(\mu ')(B)$ for every measurable subset $B$ of $U$.
Set $r = \mu(U)/\mu'(U)$. 
Then, for every subset $B$ of $U$, $\mu(B) = r \mu' (B)$ and
$$
1 = \mu(X) = \sum_k  k \mu(U_k) = \sum_k k r \mu'(U_k) = r \sum_k k \mu'(U_k) = r \mu'(X) = r .
$$
Then, there is a clopen set $B$ such that $\mu(B) \neq \mu'(B)$. 
But $$\mu(B) = \sum_{k=1}^K \sum_{j=0}^{k-1}  \mu ( B \cap T^j U_k) =
\sum_{k=1}^K \sum_{j=0}^{k-1}  \mu ( T^{-j} B \cap U_k). $$
Therefore, 
$\mu ( T^{-j} B \cap U_k) \neq 
\mu' ( T^{-j} B \cap U_k)$ for some $j,k$. 
So we get $F (\mu )(T^{-j} B \cap U_k) \neq F (\mu' ) (T^{-j} B \cap U_k)$, a contradiction.
 Hence $F$ is injective. 

Let $\nu$ be a $T_U$-invariant Borel 
probability measure on $U$. Let $\rho = \sum_{k=1}^K k \nu (U_k)$. 
For a Borel set $B \subset X$, define the $T$-invariant measure 
$$\mu(B) = \frac{1}{\rho} \sum_{k=1}^K \sum_{j=0}^{k-1}  \nu ( T^{-j}B \cap U_k).$$

We wish to see that $F (\mu)=\nu$. 
Notice that for $B \subset U$ a Borel set
$$\mu(B) = \frac{1}{\rho} \sum_{k=1}^K \nu ( B \cap U_k) = \frac{\nu(B)}{\rho} .$$

So, for $B = U$, we get $\mu (U) = 1/\rho $ and $\nu(B) = \frac{\mu(B)}{\mu(U)} = F (\mu )(B)$. Therefore, $F$ is onto. 
\end{proof}


\subsection{Observations about self-induced systems}

We will make a series of observations about self-induced systems. 
We see that 
Abramov's Formula and the Variational Principle imply that the topological entropy, denoted $h_{\rm top} (T)$, of any 
self-induced minimal Cantor system $(X,T)$ is either 0 or $\infty$.
See \cite{Petersen:1983} for an introduction to entropy and these results. 

Suppose $(X,T)$ is a self-induced minimal Cantor system. We note that 
the conjugacy to an induced system may be iterated. 
{\begin{proposition}\label{prop:recuSelfInduce}
Suppose $(X,T)$ is a minimal Cantor system  conjugate via $\varphi$ to the induced system $(U,T_U)$ where 
$U \subsetneq X$ is clopen. Then, $\varphi_{|\varphi^k(X)}$ is a conjugacy 
from  the induced system $(\varphi^k(X),T_{\varphi^k(X)})$ to   $(\varphi^{k+1}(X), T_{\varphi^{k+1}(X)})$ for any $k \geq 0$. \\ In particular $(X, T)$ and  $(\varphi^k(X),T_{\varphi^k(X)})$ are conjugate by $\varphi^k$.
\end{proposition}

\begin{proof}
The induced system of $(X,T)$ on a  clopen set $\varphi^k(X)$, is conjugate by $\varphi$ to the induced system of $(U,T_{U})$ on $\varphi^{k+1}(X)$. We have just to notice that this induced system is the induced system of $(X,T)$ on $\varphi^{k+1}(X)$.
\end{proof}
}

Using the above, we can show that 
$(X,T)$ is conjugate to an induced map $T_U$ on a clopen set $U$ where $U \subset X$ 
is an arbitrarily small set, where the smallness is expressed in terms of the invariant measures. 

\begin{proposition}
\label{prop:uniftozero}
Let $(X,T)$ be a minimal Cantor system conjugate via $\varphi$ to the induced system $(U,T_U)$ where 
$U \subsetneq X$ is clopen, 
then, $\lim_{n} \sup_{\mu \in M(X,T)} \mu \left( \varphi^n(X)\right) = 0$.
In particular, $\bigcap_{n \in \mathbb{N}} \varphi^n(X)$ has measure 0 with respect to any 
$T$-invariant Borel measure $\mu$. 
\end{proposition}

\begin{proof}
Since $U^c$ is clopen and nonempty, $\mu(U) < 1$ for all 
$\mu \in M(X,T)$, and by compactness of $M(X,T)$, $\sup_{\mu \in M(X,T)} \mu(U) < 1$.
Let $r=\sup_{\mu \in M(X,T)} \mu(U)$. 

Since $\varphi:X \to U$ is a conjugacy, every measure in $M(U,T_U)$ is of the form  
$\nu \circ \varphi^{-1}$ for some $\nu \in M(X,T)$. Therefore, by Proposition \ref{prop:inducedmeasure}, 
$\frac{1}{\mu(U)}\mu = \nu \circ \varphi^{-1}$ for some 
$\nu \in M(X,T)$. 
In particular, 
$$
r \geq \nu(U) = \nu \circ \varphi^{-1}(\varphi(U)) = \frac{\mu(\varphi(U))}{\mu(U)} .
$$
So $\mu(\varphi(U)) \leq r \mu(U) \leq r^2$.
It follows by induction that $\mu(\varphi^k(U)) \leq r^{k+1}$ for all $k \geq 1$. 
This ends the proof.
\end{proof}

\begin{proposition}
The topological entropy of any self-induced minimal Cantor system is $0$ or $\infty$. 
\label{prop:entropy}
\end{proposition}
\begin{proof}
Suppose $(X,T)$ is a minimal Cantor system and $U \subsetneq X$ is a clopen 
set such that $(U,T_U)$ is conjugate to $(X,T)$. 
Let $1> \epsilon >0$.
From propositions \ref{prop:recuSelfInduce} and \ref{prop:uniftozero}, one can take $U$ such that $\mu (U)\leq \epsilon $ for all $\mu \in M (X,T)$.

For $\mu \in M (X,T)$ we set $\mu_U = \frac{1}{\mu (U)}\mu$. 
It is a probability measure defined on the Borel sets of $U$. 
Abramov's formula relates the measure-theoretic entropy of the systems $(X,T,\mu)$ and $(U,T_U,\mu_U)$ as follows:
$h_{\mu}(T) = \mu(U) h_{{\mu}_{U}}(T_U)$. 
From the variational principle one gets 
\begin{align*}
h_{\rm top} (T) & = \sup_{\mu \in M (X,T)} h_\mu (T) = \sup_{\mu \in M (X,T)} \mu (U) h_{\mu_{U} } (T_U) \\
&  \leq  \epsilon \sup_{\mu \in M (X,T)} h_{\mu_{U} } (T_U) = \epsilon h_{\rm top} (T_U) ,
\end{align*}
where the last equality comes from Proposition \ref{prop:inducedmeasure}. Since $(U,T_U)$ is conjugate to $(X,T)$, we have $ h_{\rm top}(T_U)= h_{\rm top} (T)$,  
which ends the proof.
\end{proof}

\subsection{Bratteli-Vershik representations}
\label{subsec:bvrep}

In this section, we discuss the Bratteli-Vershik representation for a 
minimal Cantor system; every minimal Cantor system admits such a representation. 
We give a brief outline of such constructions emphasizing the notations used in this paper.
The reader can find for more details on this theory in \cite{Herman&Putnam&Skau:1992} (see also \cite{Giordano&Putnam&Skau:1995} and \cite{Durand:2010}). 
Such representations will be very helpful in many proofs or remarks below thanks to Proposition \ref{th:BVinduction}.

\subsubsection{Bratteli diagrams}

A Bratteli diagram is given by \( \left( (V_{k})_{k\ge 0}, (E_{k})_{k\ge 1}\right) \) where for each $k\ge 1$  $V_{k-1}$ is a finite set of vertices and   $ E_{k}$ is a finite set of oriented edges from $V_{k-1}$ to $V_{k}$.  The set \( V_{0} \) is a singleton \( \{v_{0}\} \), and for \( k\geq 1 \), 
$V_k=\{1,\ldots,C(k)\}$. 
For the purposes of non-degeneracy, we require 
that every vertex in \( V_{k} \) is the ``end-point'' 
of some edge in \( E_{k} \) for \( k\geq 1 \)
and an ``initial-point'' of some edge in \( E_{k+1} \) for \( k\geq 0 \). 

The \emph{level} \( k \) of the diagram is the
subgraph consisting of the vertices in \( V_{k}\cup V_{k+1} \) and the edges \( E_{k+1} \) between these vertices. 
We describe the edge set \( E_{k} \) using a \( V_{k}\times V_{k-1} \) matrix $M(k)$ over 
$\mathbb{Z}_+$ which is the adjacency matrix of the level $k$ subgraph. 
The \( (i,j) \)--entry of $M(k)$ is the number of distinct 
edges in \( E_{k} \) joining vertex \( j\in V_{k-1} \) with vertex \( i\in V_{k} \). 
For $1\leq k \leq l$ one defines 
$P(l,k)=M(l) \cdots M(k+1)$ with $P(k,k)=I$; where $I$ is the identity map. 

If for every $k$, $P(k+1,k)$ is positive, we say that the diagram $(V,E)$ is \emph{simple}. 

\subsubsection{Ordered Bratteli diagrams}

We call a {\em range located order} (RL-order for short) on a graph $(V,E)$ a partial  order on its set of edges $E$ such that two edges $e, e'$ are comparable if and only if they have the same end-point.     
An \emph{ordered} Bratteli diagram is a triple \( B=\left( (V_{k}), (E_{k}),(\preceq_{k}) \right) \) where \( \left( (V_{k}),(E_{k})\right)\) is a Bratteli
diagram and  each \( \preceq_{k}  \) is an RL-order on  the subgraph  of level $k$.


Let \( 1 \leq k\leq l \) and let \( E_{k,l} \) be the set of all paths in the graph joining vertices of \(
V_{k-1} \) with vertices of \( V_{l} \). The orders $\preceq_{k}, \ldots, \preceq_{l}$   induce an order on \( E_{k,l} \) given by 
\( \left( e_{k},\ldots,e_{l}\right) \preceq_{k,l} \left( f_{k},\ldots ,f_{l}\right)  \) if and only if there is \( k\leq i\leq l \) such that \( e_{j}=f_{j} \) for \(i<j\leq l \) and \( e_{i}\preceq_{i} f_{i} \).  Notice it is an RL-order on the graph $(V_{k}\cup V_{l}, E_{k,l})$. 

Given a strictly increasing sequence of integers \( \left( m_{k}\right) _{k\geq 0} \) with \( m_{0}=0 \) one defines the \emph{contraction} of
\( B=\left( (V_{k}),(E_{k}),(\preceq_{k}) \right)  \) (with respect to \( \left( m_{k} \right) _{k\geq 0} \)) as 
$$ \left( \left( V_{m_{k}}\right) _{k\geq 0},\left( E_{m_{k}+1,m_{k+1}}\right) _{k\geq 0},(\preceq_{m_{k} +1, m_{k+1}  })_{k\ge0} \right).$$ 
 The inverse operation is called {\it miscroscoping} (see \cite{Herman&Putnam&Skau:1992} for more details).

Given an ordered Bratteli diagram \( B=\left( (V_{k}),(E_{k}),(\preceq_k)\right) \) one defines \( X_{B} \) as the set of infinite paths \( \left( x_{1},x_{2},\ldots \right)  \) starting in \( v_{0} \) such that for all \( k\geq 1 \) the end-point of \( x_{k}\in E_{k} \) 
is the
initial-point of \( x_{k+1}\in E_{k+1} \). We topologize \( X_{B} \) by postulating a basis of open sets, namely the family of \emph{cylinder
sets}
$$
\left[ e_{1},e_{2},\ldots ,e_{l}\right]_B 
=\left\{ \left( x_{1},x_{2},\ldots \right) \in X_{B} \textrm{ } : \textrm{ }
x_{i}=e_{i},\textrm{ for }1\leq i\leq l\textrm{ }
\right\}.$$
Each \( \left[ e_{1},e_{2},\ldots ,e_{l}\right]_B  \) is also closed, as is
easily seen, and so we observe that \( X_{B} \) is a compact, totally disconnected metrizable space.

When there is a unique \( x=\left( x_{1},x_{2},\ldots \right) \in X_{B} \) such that \( x_{i} \) is maximal for any \( i\geq 1 \) and a unique
\( y=\left( y_{1},y_{2},\ldots \right) \in X_{B} \) such that \( y_{i} \) is minimal for all \( i\geq 1 \), one says that 
\( B  \) is a \emph{properly ordered} Bratteli diagram. Call these particular points \( x_{\mathrm{max}} \) and \(
x_{\mathrm{min}} \) respectively. In this case one defines a homeomorphism 
\( V_{B} : X_{B} \to X_B \) called the \emph{Vershik map} as follows. Let \( x=\left( x_{1},x_{2},\ldots \right) \in X_{B}\setminus \left\{ x_{\mathrm{max}}\right\}  \) and let \(
k\geq 1 \) be the smallest integer so that \( x_{k} \) is not a maximal edge. Let \( y_{k} \) be the successor of \( x_{k} \) and \( \left( y_{1},\ldots ,y_{k-1}\right)  \) be the unique minimal path in \( E_{1,k-1} \) connecting \( v_{0} \) with the initial vertex of \( y_{k} \). One
sets \( V_{B}\left( x\right) =\left( y_{1},\ldots ,y_{k-1},y_{k},x_{k+1},\ldots \right)  \) and \( V_{B}\left( x_{\mathrm{max}}\right)=x_{\mathrm{min}} \).

The dynamical system \( \left( X_{B},V_{B}\right)  \) is called the \emph{Bratteli-Vershik system} generated by \( B\). 
In \cite{Herman&Putnam&Skau:1992} it
is proved that any minimal Cantor system \( \left( X,T\right)  \) is conjugate to a Bratteli-Vershik system \( \left(
X_{B},V_{B}\right)  \) on a simple, properly ordered Bratteli diagram. 
One says that \( \left( X_{B},V_{B}\right) \) is a \emph{Bratteli-Vershik representation} of \( \left( X,T\right)  \).
In what follows we identify a minimal Cantor system $(X,T)$ with any of its Bratteli-Vershik representations.

A minimal Cantor system is of (topological) {\em finite rank} if it admits a  Bratteli-Vershik representation such that the number of vertices per level $\# V_n$ 
is uniformly bounded by some integer $d$. 
The minimal such $d$ is called the \emph{topological rank} of the system. We observe that topological and measure theoretical finite rank are completely different notions. For instance, systems of topological rank one correspond only to odometers whereas the Chacon subshift has  measure theoretical rank one (see \cite{Ferenczi:1997}).

\subsubsection{Induction and Bratteli-Vershik representations}

Let us show that the induction process can be easily 
``observed'' through Bratteli diagrams, as enlightened in  \cite{Giordano&Putnam&Skau:1995}.

Let $U$ be a clopen set of $(X_B , V_B)$. 
It is a finite union of cylinder sets. 
We can suppose they all have the same length {\em i.e.}, for some $n$, $U=\cup_{p\in {\texttt P}} [p]_B$ where ${\texttt P}$ is a set of paths from  $V_0$ to $V_n$.
Since the diagram is simple, we may assume that every vertex in $V_n$ is the end-point 
for a path in ${\texttt P}$. Thus, we may define a contraction $\hat{B}$ of $B$ where levels $1$ through 
$n$ are contracted to level $1$.
To obtain a Bratteli-Vershik representation of the induced system on $U$ it suffices to take the properly ordered Bratteli diagram $B'$ which on level $1$ consists of all 
edges in $\hat{B}$ corresponding to a path in ${\texttt P}$, and is otherwise the same 
as the diagram $\hat{B}$, where the new order $\preceq'_{1}$ is induced by the one of $\hat{B}$. 
It is not too much work to prove that the 
induced system on $U$ is conjugate to $(X_{B'} , V_{B'})$. 
To summarize we get the following, where the proof can be found in that of Theorem  3.8 in  \cite{Giordano&Putnam&Skau:1995}.

\begin{proposition}
\label{th:BVinduction}
Let $(X_B,V_B)$ be the Bratteli-Vershik dynamical
system associated with a simple properly ordered Bratteli diagram $B=((V_{n}),(E_{n}),(\preceq_{n}))$.
Let ${\texttt P}$ be a set of paths from $V_0$ to $V_{n_0}$.
Then, a Bratteli-Vershik representation of the induced system of $(X_B,V_B)$ on the clopen set $U=\cup_{p\in {\texttt P}} [p]_B$ is defined by the following Bratteli diagram $((V_{n}') , (E'_{n}) ,(\preceq_{n}'))$, where:
\begin{itemize}
\item The set of end-points of paths in $\texttt P$  is $V'_{1}$; 
\item
The paths ${\texttt P}$ are the edges from $V'_0 = V_{0}$ to $V'_{1}$;
\item  
The set of edges in $E_{n_{0}+1}$ with initial-points in $V'_{1}$ is $E'_{2}$; 
\item 
for $n\geq 2$, $V'_n = V_{n-1+n_0}$ and $E'_{n+1} = E_{n+n_0}$ with the induced order.
\end{itemize}
\end{proposition}

We say a Bratteli diagram $((V_{n}), (E_{n}), (\preceq))$  is {\em stationary}  whenever the sequences $(V_{n})$, $(E_{n})$ and $(\preceq_{n})$ are constant. 
Minimal Cantor systems having such a representation are exactly substitution subshifts or odometers with constant base (see Section \ref{sec:EquicontExpansif} or \cite{Durand&Host&Skau:1999}). 
A straightforward consequence of Proposition \ref{th:BVinduction} shows these systems are self-induced.

Below we make use of the Bratteli-Vershik representation to prove a technical lemma which will be useful later
in the paper. 

\begin{proposition} \label{wlog}
Let $(X,T)$ be a self-induced minimal Cantor system and let $W \subset X$ be clopen, nonempty. 
Then, $(X,T)$ is conjugate to an induced system $(U,T_U)$ where $U \subset W$. 
\end{proposition}

\begin{proof}
Let $(X,T)$ be a self-induced minimal Cantor system and let $W \subset X$ be 
a nonempty clopen set. By iterating the self-induction conjugacy $\varphi$ sufficiently 
many times, Proposition \ref{prop:uniftozero} implies that $(X,T)$ is conjugate to an induced map $(U,T_U)$ where
$\mu(U) < \mu(W)$ for all $\mu$ in $M(X,T)$. 

For large enough $n$, there exist sets, ${\texttt P}_U$ and ${\texttt P}_W$,  of paths ending at level $n$ satisfying $U=\cup_{p\in {\texttt P}_U} [p]_B$ and $W=\cup_{p\in {\texttt P}_W} [p]_B$. 
Since $\mu(U) < \mu(W)$ for all $\mu$ in $M(X,T)$, 
taking $n$ sufficiently large, we may assume that for each vertex $v \in V_n$, 
the number of paths in $\texttt P_U$ which end at $v$ is less than the number of paths in 
$\texttt P_W$ which end at $v$ (This is not a trivial fact, but is rather standard by now. 
It follows from the 
simplicity of the diagram and the Ergodic Theorem. For example see proof of 
Lemma 2.5 in \cite{Glasner&Weiss:1995}).

For every vertex $v$ in $V_n$, let
$u(v)$ equal the number of $\texttt P_U$-paths ending at $v$ and 
$w(v)$ the number of $\texttt P_W$-paths ending at $v$. By assumption $u(v)<w(v)$ for each $v$. 

Define a set $W' \subset W$ which is the union over $v \in V_n$ of the 
union of cylinder sets of the first $u(v)$ $\texttt P_W$-paths ending at $v$.
Then, we have Bratteli-Vershik representations of the induced map on $U$ and the induced map on 
$W'$ on diagrams which both have exactly $u(v)$ edges from $v_0$ to $v \in V_1$ 
and are identical below level $n$. 
Thus, $(U,T_U)$ is conjugate to $(W',T_{W'})$ where $W' \subset W$. 
\end{proof}

\section{Equicontinuous and expansive cases}\label{sec:EquicontExpansif}

We recall that, from the previous section, to be self-induced it is 
necessary to have zero entropy or infinite entropy. 
Among zero entropy dynamical systems is the family of those having finite topological rank. 
A result of Downarowicz and Maass \cite{Downarowicz&Maass:2008} asserts that in this family the Cantor systems are either equicontinuous, that is odometers, or expansive, that is subshifts. 
Here we treat these both cases characterizing the self-induced Cantor systems of finite topological rank.


\subsection{Equicontinuous case: the odometers}


Let $(q_n)_{n\ge 1}$ be a integer sequence such that $q_n \geq 2$.
We set  $p_{n} = q_{n}\cdots q_2 q_1$.
The set of $(q_n)$-{\em adic integers} is the inverse limit
$$
\mathbb{Z}_{(q_n)} = \left\{ (x_n) \in \prod_{n=1}^{\infty} \mathbb{Z} / p_n \mathbb{Z} : x_{n} \equiv x_{n+1} \mod p_n \right\}.
$$
We endow $\prod_{n=1}^{\infty} \mathbb{Z} / p_n \mathbb{Z}$ with the product topology of the discrete topologies. The set of $(q_{n})$-adic integers $\mathbb{Z}_{(q_n)}$ endowed with the induced topology forms  a Cantor set. It is a  topological group (Exercise). The dual group of $\mathbb{Z}_{(q_n)}$ ({\em i.e.}, the group of continuous characters) is conjugate to the group of $(q_{n})_{n}$-adic rationals

$$
{\mathbb Q}_{(q_{n})} = \left\{ \frac j{p_{n}} :  j\in \mathbb{Z} , n\geq 1 \right\}.
$$ 

A base of the topology  of $\mathbb{Z}_{(q_n)}$  is given by the sets 
$$
[a_1 , a_2 , \dots , a_n] = \{ (x_n) \in \mathbb{Z}_{(q_n)} : x_i = a_i , 1\leq i\leq n  \},
$$
with $ a_{i} \in  \mathbb{Z} / p_i \mathbb{Z}$, $i=1 , \ldots, n$.

When $p_n = p^n$ for all $n$, it defines the classical ring of $p$-adic integers $\mathbb{Z}_p$.
Let $R : \mathbb{Z}_{(q_n)} \to \mathbb{Z}_{(q_n)}$ be the map $x\mapsto x+1$.
As it will always be clear from the context these maps will be denoted by $R$ for any sequence $(q_n)$.
The Cantor system  $(\mathbb{Z}_{(q_n)} , R)$ is called {\em odometer in base} $(p_n)$ or {\em with characteristic sequence} $(q_n)$. 
This system is minimal and aperiodic.

Following the notation of Section \ref{subsec:bvrep}, a Bratteli-Vershik representation of the odometer $(\mathbb{Z}_{(q_n)} , R)$ is obtained taking a single vertex at each level, $\# V_n = 1$, $n\geq 0$, and setting $q_{n+1}$ edges in $E_{n+1}$.
As there is a unique vertex at each level, the ordering on the $E_n$'s is arbitrary. 
Consequently, odometers are of topological rank 1 and it is easy to prove that rank one Cantor systems are odometers.
See \cite{Durand:2010} for more details. 
We say it is a {\it one-vertex} Bratteli diagram.

Let us recall that a minimal Cantor system $(X,T)$ which is {\em equicontinuous} ({\em i.e.}, the family of maps $\{T^n\}_{n \in {\mathbb Z}}$ is equicontinuous) is conjugate to an odometer, and, the  two following well known results  (see \cite{Kurka:2003}). 
%
%

\begin{lemma}\label{lem:factorodometer}
The following are equivalent :
\begin{enumerate}
\item
The odometer $({\mathbb Z}_{(q_{n}')} , R)$ is a factor of the odometer $({\mathbb Z}_{(q_{n})}, R)$;
\item
${\mathbb Q}_{(q_{n}')}$ is included in ${\mathbb Q}_{(q_{n})}$;
\item
For all $n$ there exists $k$ such that $p_{n}'$ divides $p_{k}$.
\end{enumerate}
\end{lemma}

\begin{lemma}
\label{lemma:conjugate}
The following are equivalent.
\begin{enumerate}
\item
$({\mathbb Z}_{(q_{n}')} , R)$ and $({\mathbb Z}_{(q_{n})}, R)$ are conjugate.
\item
${\mathbb Q}_{(q_{n}')} = {\mathbb Q}_{(q_{n})}$.
\end{enumerate}
\end{lemma}

This last condition is true if and only if for all prime number $p$ we have
\begin{align}
\label{relation:qpn}
\lim_n \max \left\{ k : p^k \hbox{ divides } p_n  \right\} = \lim_n \max \left\{ k : p^k \hbox{ divides } p_n'  \right\} .
\end{align}

As a corollary, we obtain the following. The proof is left to the reader.

\begin{corollary}
\label{coro:odoprime}
Let  $({\mathbb Z}_{(q_{n})} , R)$ be an odometer.
\begin{enumerate}
\item
Any permutation of the elements of the characteristic sequence will define the same odometer, up to conjugacy.
\item
$({\mathbb Z}_{(q_{n})},R)$ is conjugate to the odometer with characteristic sequence $(s_n)$ where $s_n$ are prime numbers such that $p_{i+1}/p_i = s_{n_i} s_{n_i+1} \cdots s_{n_{i+1}-1}$ for some strictly increasing sequence $(n_i)$ with $n_0=1$.
\end{enumerate}
\end{corollary}

Thanks to these results we deduce a characterization of self induced odometers. 
The proof will use Bratteli diagrams. 
One can also prove this using the periodic structure of $\mathbb{Z}_{(q_n)}$. This would be, in some sense, easier if one is more familiar with $(q_n)$-adic groups than Bratteli diagrams.
However, we want to emphasize how useful Bratteli diagrams can be when studying Poincar\'e recurrence problems.

\begin{proposition}
\label{prop:chara_selfinduced_odo}
Let $({\mathbb Z}_{(q_{n})}, R)$ be an odometer. The following are equivalent.

\begin{enumerate}

\item
\label{item:TSI}
$({\mathbb Z}_{(q_{n})}, R)$ is self-induced.
\item
\label{item:primes}
There exists a prime number $q$ dividing infinitely many $q_n$, with $q_n=p_{n+1}/p_n$, {\em i.e.}, $\lim_n \max \left\{ k : q^k \hbox{ divides } p_n  \right\} = \infty $. 
\item
\label{item:homo}
There is a homomorphism from ${\mathbb Z}_{(q_{n})}$ to $\mathbb{Z}_p$  for some prime number $p$.
\item
\label{item:factorodo}
$(\mathbb{Z}_p , R )$ is a factor of $({\mathbb Z}_{(q_{n})} , R)$. 
\end{enumerate} 
\end{proposition}

\begin{proof}
We leave as an exercise to prove that \eqref{item:primes}, \eqref{item:homo} and \eqref{item:factorodo} are equivalent.
We prove \eqref{item:primes} and \eqref{item:TSI} are equivalent.


\eqref{item:primes} $\implies$ \eqref{item:TSI}
Again using Corollary \ref{coro:odoprime}, we may assume $p_1=q_1=q$. 
Then, Lemma \ref{lemma:conjugate} together  with Equality  \eqref{relation:qpn} imply   that the odometer
with characteristic sequence $(q_n)_{n \geq 1}$ is conjugate to the odometer with characteristic sequence
 $(q_{n+1})_{n \geq 1}$.

Let $B$ be the one-vertex Bratteli diagram with $q_n$ edges in $E_n$ for $n \geq 1$. 
Select a single edge $e$ from $E_1$ and let $U=[e]$ ({\em i.e.}, the clopen set of infinite paths 
starting with $e$).
 Proposition \ref{th:BVinduction} gives us that the induced system on $U$ is conjugate to an odometer with characteristic sequence $(q_{n+1})$.
Thus, $({\mathbb Z}_{(q_{n})}, R)$ is self-induced.

\eqref{item:TSI} $\implies$ \eqref{item:primes}
Suppose $({\mathbb Z}_{(q_{n})}, R)$ is self-induced.
There exists a proper clopen set $U$ such that $(U, R_U)$ is conjugate to $({\mathbb Z}_{(q_{n})}, R)$.
Since $(U, R_U)$ is an induced system of $({\mathbb Z}_{(q_{n})}, R)$, Proposition \ref{th:BVinduction} implies $(U, R_U)$ is an odometer with characteristic sequence $(q_n')$ where, for all $n\geq 2$, $q_n' = q_{n+n_0-1}$, for some $n_0$, and $q_{1}' < q_1 q_2 \cdots q_{n_0}$.

By contradiction, let us assume that there is no prime number $p$ such that $\lim_n \max \left\{ k : p^k \hbox{ divides } q_1q_2 \cdots q_n  \right\} = \infty$. 
Then, from Lemma \ref{lemma:conjugate}, for all prime numbers $p$, we should have 
\begin{align*}
& \lim_n \max \left\{ k : p^k \hbox{ divides } q_1 q_2 \cdots q_n  \right\} \\
= & \lim_n \max \left\{ k : p^k \hbox{ divides } q_1' q_{n_0+1} \cdots q_{n+n_0-1}  \right\} < \infty .
\end{align*}
This is  in contradiction with the fact that $q_{1}'$ is strictly less than $ q_1 q_2 \cdots q_{n_0}$.
\end{proof}

\subsection{Expansive case: the substitution subshifts}

In this section we prove that a self-induced minimal Cantor system $(X,T)$ 
is expansive if and only if it is conjugate to a substitution subshift on a finite alphabet. 
This result may be of interest on its own, it 
generalizes a similar result in the case where the induced map is of the form $T^k$ for some fixed $k$ \cite{Durand:thesis}.

Recall that a homeomorphism $T:X \to X$ on a compact metric space 
$(X,d)$ is \emph{expansive} if there exists a constant 
$\delta >0$ such that given any $x,y \in X$,  $x \neq y$, there
exists an $n \in \mathbb{Z}$ such that $d(T^nx,T^ny) > \delta$. A classical theorem 
states that if $X$ is a Cantor space and $(X,T)$ is an expansive
dynamical system then, $(X,T)$ is conjugate to a subshift. 

\subsubsection{Subshifts and induction}
\label{subsec:inducedsubshifts}

Let $\mathcal{A}$ be an {\em alphabet}, that is, a finite set.
We endow $\mathcal{A}^{\mathbb{Z}}$ with the infinite product of the discrete topology
and consider the shift map $S: \mathcal{A}^{\mathbb{Z}} \to \mathcal{A}^{\mathbb{Z}}$. 
That is, for a sequence 
$x = (x_i) \in \mathcal{A}^{\mathbb{Z}}$,  $S(x)$ is the sequence defined by $S(x)_i = x_{i+1}$. 
By a {\it subshift} on $\mathcal{A}$ we shall mean a couple $(X,S)$ where 
$X \subset \mathcal{A}^{\mathbb{Z}}$ is closed and shift-invariant ($S(X) = X$).
See \cite{Lind&Marcus:1995} for a thorough introduction to subshifts. 

A {\em word} is an element of the free monoid $\mathcal{A}^{*}$ generated by $\mathcal{A}$ where the neutral element is denoted by $\epsilon$ and is called the {\em empty word}. 
For $w = w_1 \ldots w_n \in \mathcal{A}^*$, $w_i \in \mathcal{A}$, we use 
$|w|$ to denote the {\em length} of $w$, that is $|w| =n$.
We denote by $\mathcal{A}^n$ the set of words of length $n$.
We set $\mathcal{A}^{+} = \mathcal{A}^{*} \setminus  \{ \epsilon \}$.

For a point $x = (x_i) \in \mathcal{A}^{\mathbb{Z}}$, and $i,j \in \mathbb{Z}$ with $i\leq j$ 
let $x[i,j]$ denote the word $x[i,j] = x_i \ldots x_j$.
We say $i$ is an occurrence of the word $u$ in $x$ if $u=x[i,i+|u|-1]$.
We keep the same notation for finite words.

For a subshift $(X,S)$, we call {\em language} of $X$ the subset of $\mathcal{A}^*$ defined by 
$\mathcal{L} (X) = \{ x[i,j] : x \in X , i\leq j\}$.
We set $\mathcal{L}_k(X) = \{ x[i,i+k-1] : x \in X , i\in \mathbb{Z} \}$.
A base for the subspace topology on $X$ consists of the family of the sets
$$
[w.v]_X = \{ x\in X : x[-|w|,|v|-1] = wv \},
$$
where $w,v \in \mathcal{A}^*$.
These sets $[w.v]_X$ are called {\it cylinders}. 
When it will not create confusion we will write $[w.v]$ instead of $[w.v]_{X}$. 
We set $[v] = [\epsilon . v]$, where $\epsilon$ is the empty word of $\mathcal{A}^*$.

Every shift commuting continuous function $F : X \to Y$ between two subshifts is given 
by a \emph{sliding block code} \cite{Lind&Marcus:1995}. 
That is, a map $f$ from the set of words of length $2r+1$ in $X$ to 
letters appearing in elements of $Y$ such that for all $x \in X$ and $i \in \mathbb{Z}$, 
$F (x)_i = f (x[i-r,i+r])$.  

Suppose $X \subset \mathcal{A}^{\mathbb Z}$ is a subshift. 
The \emph{$k$-block presentation of $X$} is the subshift 
$X^{[k]}$ with alphabet $\mathcal{A}^k$ defined as the set of the sequences $(w_i) \in \left(\mathcal{A}^k \right)^{\mathbb{Z}}$ such that
\begin{enumerate}
\item
each $w_i=w_{i,1}w_{i,2}\ldots w_{i,k}$ is in $\mathcal{L}_k (X)$,
\item 
$w_{i,2}\ldots w_{i,k}=w_{i+1, 1}\ldots w_{i+1 , k-1}$ for all $i \in \mathbb{Z}$,
\item 
the sequence $(w_{i,1})$ is in $X$. 
\end{enumerate} 
It is well-known that the shift map on the $k$-block presentation of $X$ is conjugate to the shift map on $X$
 \cite{Lind&Marcus:1995}. 

Let us explain a way to see the induction process via words in minimal subshifts.
Let $(X,S)$ be a minimal subshift defined on the alphabet $\mathcal{A}$ and 
$U$ be a clopen subset of $X$.
Then, there exist words $u_1^-, \dots , u_n^-, u_1^+, \dots u_n^+$ such that 
$U = \cup_{i=1}^n [u_i^- . u_i^+]$.
Observe that these words can be chosen of equal lengths.

Let $\mathcal{R}_{X,U} \subset \mathcal{L}(X)$ 
be the set of {\em return words to} $U$, that is, the set of words
$x[i,j-1]$, with $x\in X$ and $i<j$, where 
\begin{enumerate}
\item
$S^ix$ belongs to $U$;
\item
$S^j x$ belongs to $U$;
\item
there is no integer $k$, with $i < k < j$, such that $S^k x$ belongs to $U$.
\end{enumerate}
Since the subshift $(X,S)$ is minimal then, the length of the return words  to $U$ is finite and
the set of return words to $U$ is also finite.
 
When $U$ consists of a single word and $u^-$ is the empty word, then we recover the classical notion of return words (see \cite{Durand:thesis,Ferenczi&Mauduit&Nogueira:1996,Durand:1998}).
In this situation, $U = [\epsilon . u^+]$, the return words $w$ are characterized by the fact that 
$ u^+$ is a prefix and a suffix of $wu^+$, and, $wu^+$ has no other occurrence of $ u^+$ than these two (see \cite{Durand&Host&Skau:1999}).

The general situation is a bit different. 
We should keep the information of which $u_i^\pm$ precedes and follows a return word $w$, for each occurrences of $w$: $u_i^- w u_j^+$ and $u_k^- w u_l^+$ could belong to $\mathcal{L} (X)$ for different $i,j,k,l$.  
Thus, this return word $w$ represents two different ``returns'' to $U$ (because it returns to different cylinders describing $U$).
To this end we set
$$
\tilde{\mathcal{R}}_{X,U} = \{ (u^- , w,u^+) : u^-wu^+\in \mathcal{L} (X) , w \in {\mathcal{R}}_{X,U}, u^-\in U^-, u^+ \in U^+ \} , 
$$ 
where $U^- = \{ u_1^- , \dots , u_n^- \}$ and $U^+ = \{ u_1^+ , \dots , u_n^+ \}$.

Let ${\theta}$  be the morphism from  $\tilde{\mathcal{R}}_{X,U}^*$ to $\mathcal{A}^*$ defined by $\theta  (u^- , w,u^+) = w$. 
As usual, we extend by concatenation $\theta$ to a map, also denoted $\theta$, from  $\tilde{\mathcal{R}}_{X,U}^\mathbb{Z}$ to $\mathcal{A}^\mathbb{Z}$.
It is easy to check that for each $x\in U$ there is a unique $y\in \tilde{\mathcal{R}}_{X,U}^{\mathbb Z}$ such that $x=\theta (y)$. 
Let $Y$ be the set of all these $y$. 
It is invariant under the shift and hence $(Y,S)$ is a subshift on the alphabet $\tilde{\mathcal{R}}_{X,U}$.

We left as an exercise to prove that $\theta : (Y,S) \to (U,S_U)$ is a conjugacy map. 
For the ideas to use we refer to \cite{Durand&Host&Skau:1999} where this is treated when $U = [u^- . u^+]$.

\subsubsection{Substitutive and primitive substitution subshifts}

We recall here a small am\-ount of the background necessary for this section on 
minimal substitution subshifts. For a much more in depth introduction 
see \cite{Durand&Host&Skau:1999,Fogg:2002, Queffelec:1987}.

Suppose we have an alphabet $\mathcal{A}$ and a map $\sigma:\mathcal{A} \to \mathcal{A}^+$. 
As usual, a substitution $\sigma$ on an alphabet space $\mathcal{A}$ can be extended to
functions $\sigma:\mathcal{A}^{+} \to \mathcal{A}^{+}$ by concatenation, 
and therefore iterated. 

Because we wish to focus on minimal systems, we will assume that all of our
substitutions are \emph{primitive} {\em i.e.}, the following two properties hold. 
\begin{itemize}
\item there is an $n \in \mathbb{N}$ such that for any two letters $a,b \in \mathcal{A}$, 
the letter $b$ appears in the word $\sigma^n(a)$,
\item there is an $a \in \mathcal{A}$ such that $\lim_{k \to \infty}|\sigma^k(a)| = \infty$. 
\end{itemize}

Given a primitive substitution $\sigma$, let $X_{\sigma}$ denote the set of sequences $x$ in 
$\mathcal{A}^{\mathbb{Z}}$ with the property that every word that appears in $x$ is a 
subword of $\sigma^k(a)$ for some $a \in \mathcal{A}$ and $k$. 
This set is $S$-invariant.
As it will not create confusion, we continue to denote by $S$ the restriction of $S$ to $X_\sigma$.
The couple $(X_\sigma , S)$ is called {\em primitive substitution subshift} (see \cite{Queffelec:1987} for more details on these subshifts).

In the literature one can find less restrictive definitions of substitution subshifts that are, up to conjugacy, irrelevant in the minimal context as shown by Proposition \ref{prop:sub=sub} below. 
For instance, let us present a general definition of {\em substitutive subshifts}.
Let $\sigma $ be an endomorphism of the free monoid $\mathcal{A}^{*}$ and let $\phi : \mathcal{A}^{*} \to \mathcal{B}^{*}$, for some finite alphabet $\mathcal{B}$.
Let $X$ be the set of infinite sequences $x$ in $\mathcal{B}^{\mathbb{Z}}$ with the property that every word that appears in $x$ is a 
subword of $\phi \circ \sigma^k(a)$ for some $a \in \mathcal{A}$ and $k$. 
It could happen that $X$ is empty but it is a closed $S$-invariant subset of $\mathcal{B}^{\mathbb{Z}}$.
We say $(X,S)$ is a {\it substitutive subshift}.
Within the category of minimal subshifts, the following result shows that our primitiveness hypothesis on the substitution is not restrictive.
In fact, it proves that if you are given a non-primitive substitution generating a minimal subshift $(X,S)$ then, there is a primitive substitution whose subshift is conjugate to $(X,S)$.
 
\begin{proposition}
\cite{Durand&Host&Skau:1999,Durand:2013}
\label{prop:sub=sub}
Let $(X,S)$ be a minimal subshift.
The following are equivalent.
\begin{enumerate}
\item
$(X,S)$ is conjugate to a substitutive subshift.
\item
$(X,S)$ is conjugate to a primitive substitution subshift.
\end{enumerate}
\end{proposition}
For example the Chacon substitution $0\mapsto 0010$, $1\mapsto 1$ is not primitive but its subshift $(X,S)$ is minimal and it is conjugate to a primitive substitution subshift (see \cite{Durand:thesis}).

The following is well-known, e.g. see \cite[Cor 12]{Durand&Host&Skau:1999}.
It is a direct consequence of Moss\'e's result on recognizability (see \cite{Mosse:1992, Mosse:1996}).

\begin{proposition}\label{prop:SubSelf}
Aperiodic primitive substitution subshifts are self-induced minimal Cantor systems. 
\end{proposition}
Our goal here is to prove the converse, that expansive self-induced minimal Cantor systems 
must be substitution subshifts. 
Observe that self-induced minimal Cantor systems are necessarily aperiodic.
\begin{theorem} \label{submain}
Let $(X,S)$ be a self-induced expansive minimal Cantor system. 
Then, it is conjugate to a substitution subshift.
\end{theorem}
\begin{proof}
Assume that $(X,S)$ is an expansive minimal Cantor system (hence a subshift on some alphabet $\mathcal{A}$, up to conjugacy) that is self-induced.
This means there is a clopen set $U\subsetneq X$ such that $(X,S)$ and $(U,S_U)$ are conjugate.
We set $U = \cup_{i=1}^N [u_i^-.u_i^+]$ for some words $u_i^\pm$.
Without loss of generality we can suppose that $|u_i^+|=|u_j^-|$, for all $i,j$.

Taking the notations of Section \ref{subsec:inducedsubshifts} and the observation we made, $(U,S_U)$ is isomorphic to some subshift $(Y,S)$ defined on the alphabet $\mathcal {R}=\tilde{\mathcal{R}}_{X,U}$.
Notice that the map $\theta$ defined in this section is such that the length $|\theta (a)|$ is greater or equal to $1$ for all $a\in \mathcal {R}$.

Let $f: \mathcal{A}^{2r+1} \to \mathcal{R}$ be a sliding block code defining a conjugacy from $(X,S)$ to $(Y,S)$.
We claim there exists an integer  $l$ such that
\begin{align}
\label{ineg}
|\theta (y_1 \dots y_l )| \geq r+l 
\end{align}
for all words $y_1 \dots y_l$ of $\mathcal{L}(Y)$.
Indeed, $(Y,S)$ being minimal it suffices to show there exists a letter $a\in \mathcal{R}$ such that $|\theta (a)| \geq 2$.
Suppose it is not the case. 
This means that $S^k x $ belongs to $U$ for all $k$ and $x\in X$.
Thus, from the minimality of $(X,S)$, we obtain that $U=X$ which contradicts the choice of $U$.

For all $n$, let $\mathcal{A}_n$ be the alphabet 
$$
\{ (x_{-n} \ldots x_{-1} x_0 x_1 \ldots x_{n}) : x_{-n} \ldots x_{-1} x_0 x_1 \ldots x_{n} \hbox{ is a word of } X \} .
$$

Recall that $X^{[2n+1]}$ denote the $(2n+1)$-block presentation of $X$, that is
$$
X^{[2n+1]} = \{ ((x_{k-n} \ldots x_{k-1} x_k x_{k+1} \ldots x_{k+n}))_k : (x_k)_k \in X \}.
$$

Of course the subshift $(X,S)$  is conjugate to $(X^{[2n+1]} , S)$, by the sliding block code $h_{n} \colon \mathcal{A}^{2n+1} \to \mathcal{A}_{n}$  defined as $h_{n}(x_{-n}\ldots x_{n}) = (x_{-n}\ldots x_{n})$.
In the same way we can define $\mathcal{R}_n$ and $Y^{[2n+1]}$.  The map $f$ naturally defines a sliding block code $g : \mathcal{A}_{l+r} \to \mathcal{R}_l$ asserting

\begin{align*}
g ( (x_{-(l+r)} \ldots x_{-1} x_0 x_1 \ldots x_{l+r})) =& (f (x_{-(l+r)} \ldots x_{-1} x_0 x_1 \ldots x_{l+r})) \\
= &(y_{-l} \ldots y_{-1} y_0 y_1 \ldots y_{l}) .
\end{align*}

This provides a conjugacy from $(X^{[2(r+l)+1]}, S)$ to $(Y^{[2l+1]} ,S)$.
To conclude, it suffices to prove that $(Y^{[2l+1]}, S)$ is a substitution subshift. 
For this, let us construct a morphism $\tau : \mathcal{R}^*_l \to \mathcal{A}_{l+r}^*$.   

For each  $b=(y_{-l} \ldots y_{-1}y_0 y_1 \dots y_{l})$ in $\mathcal{R}_{l}$, we set, for $ -l \le k \le l$,
$\theta(y_{k}) = x_{0}^{(k)}\ldots x_{|\theta(y_{k})|-1}^{(k)} \in \mathcal{A}^*$. So we have
$$ \theta(y_{-l} \dots y_{-1}y_0 y_1 \dots y_{l}) = x_{0}^{(-l)}\ldots x_{|\theta(y_{-l})|-1}^{(-l)} \ldots   x_{0}^{(l)}\ldots x_{|\theta(y_{l})|-1}^{(l)}.$$

Observe that by \eqref{ineg}, the length of the word     $\theta(y_{-l} \dots y_{-1}y_0 y_1 \dots y_{l})$, $\sum_{k=-l}^l |\theta(y_{k})|$ is greater than $2(r+l) +1$. 
Let $p$ (resp. $s$) be the prefix (resp.  the suffix) of  $\theta(y_{-l} \dots y_{-1}y_0 y_1 \dots y_{l})$ of respective  lengths   $\sum_{k=-l}^{-1} |\theta(y_{k})| -l-r$ and $\sum_{k=1}^l |\theta(y_{k})| -l-r$.  Let $w$ be the word of length $|\theta(y_{0})| +2 (l+r) $ such that $\theta(y_{-l} \dots y_0\dots y_{l}) =  pws$.
Finally we define 
$$
\tau (b) = h_{l+r} (w) \in \mathcal{A}_{l+r}^*,$$
where $h_{l+r}$ denote the sliding block code  $\mathcal{A}^{2(l+r)+1} \to \mathcal{A}_{l+r}$ .

It is plain to check that for every word $y_{i-l}y_{i-l+1} \ldots y_{j+l}$ in $\mathcal{L}(Y)$, with $i \le j$, the word $\tau((y_{i-l}\ldots y_{i} \ldots y_{i+l})(y_{i-l+1}\ldots y_{i+l+1})  \ldots (y_{i-l} \ldots  y_{j}\ldots y_{j+l}))$ is a factor of the word $h_{l+r}(\theta(y_{i-l}y_{i-l+1} \ldots y_{j+l}))$. 

Consequently $\tau $ sends words of $\mathcal{L} (Y^{[2l+1]})$ to words of $\mathcal{L}(X^{[2(r+l)+1]})$ and $g\circ \tau$ defines an endomorphism of $\mathcal{A}_{l}^*$ preserving the langage $\mathcal{L}(Y^{[2l+1]})$. 

We claim that for some $a\in \mathcal{R}_l$ the sequence of lengths $(| (g\circ \tau )^n (a)|)_n$ goes to infinity.
This claim enables us to conclude because it implies that  $g\circ \tau$ defines a non-empty substitution subshift included in $Y^{[2l+1]}$.
From the  minimality of $(Y^{[2l+1]}, S)$, both subshifts  coincide and we prove the theorem using Proposition \ref{prop:sub=sub}.

Let us prove the claim by contradiction. 
Thus, for all words $w\in \mathcal{R}_l^*$ the sequence $(|(g\circ \tau )^n (w)|)_n$ is bounded.
Let $a$ be a letter of $\mathcal{R}$ satisfying $|\theta (a)|\geq 2$ and $b = (y_{-l} \dots y_{-1}y_0 y_1 \dots y_{l}) \in \mathcal{R}_{l}$ where $y_0 = a$.
The system $(Y^{[2l+1]} , S)$ being minimal, there exists a constant $K$ such that any word $w\in \mathcal{R}_{l}^*$ of length $K$ has an occurrence of $b$.
Let $w$ be such a word.
It has an occurrence of $b$ and any other letter $c$ appearing in $w$ satisfies $|\theta (c)| \geq 1$. 
Thus, $|g\circ \tau  (w)|\geq |w|+1>K$ and consequently for all $n\ge 1$, $|(g\circ \tau )^{n} (w)| \geq |w| + n$.
We get a contradiction and this shows the claim.
\end{proof}


\subsection{The finite topological rank case}


We recall that the topological rank of a minimal Cantor system $(X,T)$ is the minimum over all its Bratteli-Vershik representations $B = ((V_{n}),(E_{n}), (\preceq_{n}) )$ of the quantity $\sup_n \# V_n$.
For example odometers have topological rank 1, Sturmian subshifts have topological rank 2 (see \cite{Dartnell&Durand&Maass:2000}), minimal substitution subshifts all have a finite topological rank (see \cite{Durand&Host&Skau:1999}), as the Cantor version of interval echange transformations (see \cite{Gjerde&Johansen:2002}). It is well-known  minimal Cantor systems with finite topological rank always have zero entropy. Let us mention that
Toeplitz subshifts can have finite or infinite topological rank as some are substitution subshifts and some have positive topological entropy.

In \cite{Downarowicz&Maass:2008} the authors proved the following nice result.

\begin{theorem}\cite{Downarowicz&Maass:2008}
A minimal Cantor systems with finite topological rank is necessarily expansive or equicontinuous. 
\end{theorem}

Thus, from the previous sections we deduce the following result.

\begin{corollary}
\label{theo:mainsub}
A minimal Cantor systems with finite topological rank is self-induced if and only if it is a minimal substitution subshift or an odometer fulfilling Property \eqref{item:primes} of Proposition \ref{prop:chara_selfinduced_odo}.
\end{corollary}


\section{Self-induced examples that are neither equicontinuous nor expansive}\label{sec:nonEquicnonExp}


In this section, we provide examples of self-induced minimal Cantor systems that are non-expansive and non-equicontinuous. 
The first one has  zero topological entropy whereas the second has an infinite topological entropy. To present them, we first recall the notion of {\em Toeplitz subshift}.

\subsection{Toeplitz subshift} 
A {\em Toeplitz sequence} on a finite alphabet $\mathcal A$ is a sequence $(x_{n})_{n\in \mathbb Z} \in {\mathcal A}^{\mathbb Z}$ such that $$\forall n \in \mathbb Z, \exists p \in \mathbb Z \hbox{ such that } x_{n+kp} = x_{n}, \forall k\in \mathbb Z.$$    
A  subshift $(X,S)$ is Toeplitz when  $X= \overline{\{S^{i}(x): i\in \mathbb Z\}}$ for a Toeplitz sequence $x$.  Let us recall some properties we need. We refer to \cite{DToep}  for a survey on Toeplitz subshifts. Any Toeplitz subshift is minimal and any  aperiodic Toeplitz subshift $(X,S)$ is an almost one-to-one extension of an odometer: {\em i.e.}, there exists a point with a unique preimage. Hence this odometer is the {\em maximal equicontinuous} factor of $(X,S)$, meaning that any equicontinuous factor of $(X,S)$ is a factor of this odometer.

\subsection{Self induced example with zero entropy}
\label{sec:noneqnonexp}

Our example is a product of a substitution subshift and 
an odometer. 
Let $\sigma$ be the substitution defined on the alphabet $\{ 0,1 \}$ by 
 
$$ 
\sigma(0) = 01 \textrm{ and   } \sigma (1)=00.
$$

Let $(X_{\sigma}, S)$ be the subshift it generates.
From a well-known result of Dekking \cite{Dekking:1978}, this is a minimal Cantor system and its maximal equicontinuous factor 
is the odometer defined on $\mathbb{Z}_2$.

Let consider the Cantor system $(X_{\sigma} \times \mathbb{Z}_{3} , T )$ defined by
\begin{eqnarray*}
T \colon X_{\sigma} \times \mathbb{Z}_{3} & \to & X_{\sigma} \times \mathbb{Z}_{3} \\
(x,z) & \mapsto & ( Sx , z+1).
\end{eqnarray*}

\begin{property}
The system $(X_{\sigma} \times \mathbb{Z}_{3} , T )$ is minimal.
\end{property}
Even if  it can be checked by using the very regular structure of the systems, we use a general result on Toeplitz systems in \cite{DToep} to prove it. 
\begin{proof}
From Lemma 12.1 in \cite{DToep} and since $(X_{\sigma},S)$ does not have a topological periodic factor of period 3, we deduce that the  $(X_{\sigma}, S)$ and $( \mathbb{Z}_{3} , z\mapsto z+1 )$ are measurably disjoint. Hence  the  product system has the product measure as
unique invariant probability measure, implying the minimality of the system. \end{proof}

\begin{property}
The system $(X_{\sigma} \times \mathbb{Z}_{3} , T )$ is neither expansive nor equicontinuous.
\end{property}
\begin{proof}
Let $d_1$ and $d_2$ be two distances defining the topologies of 
$X_\sigma$ and $\mathbb{Z}_3$, respectively.
The action of $T$ is not equicontinuous as it has an expansive factor, namely $(X_{\sigma}, S)$.

Let us show it is not expansive. Let $\epsilon>0$ be arbitrary, 
it is enough  to construct two points 
$(x,z) \neq (x',z') \in X_{\sigma} \times \mathbb{Z}_{3}$
such that $d_1(S^nx,S^nx') < \epsilon$ and 
$d_2(z+n,z'+n) < \epsilon$ for all integer $n$. 
Because the odometer $\mathbb{Z}_3$ is an isometry, 
we simply select any $x=x'$ and any $z \neq z'$ with 
$d_2(z,z') < \epsilon$. 
\end{proof}

Notice that the two previous claims also hold for the product of an  odometer  $(\mathbb{Z}_q , z\mapsto z+1)$ with a Toeplitz subshift having   $ (\mathbb{Z}_p , z \mapsto z+1)$ as maximal equicontinuous factor,  when $p$ and $q$ are coprime.

 \begin{property}
The minimal Cantor system $(X_{\sigma} \times \mathbb{Z}_{3}, T)$ is self-induced. 
 \end{property}

\begin{proof}
Define $\varphi: X_{\sigma} \times \mathbb{Z}_{3} \to \sigma(X_{\sigma}) \times \mathbb{Z}_{3}$ by 
$$
\varphi(x,z) = (\sigma(x),2 z) .
$$
The map $\sigma$ from $X_{\sigma}$ onto its image 
satisfies 
$S^2 \circ \sigma = \sigma \circ S$. 
Moreover, since the substitution is injective on the letters and of constant length,  the map $\sigma : X_\sigma \to \sigma (X_\sigma )$ is one-to-one. Thus,  the return time for the map $S$ to the clopen set  $\sigma(X_{\sigma})$ is constant and equal to 2.  
Let $U$ be the clopen set $U = \sigma(X_{\sigma}) \times \mathbb{Z}_{3}$.
Then, the induced map is 
$$ 
T_{U} : (x,z) \mapsto (S^2 x, z+2).
$$

Furthermore, we have 
$$T_U \circ \varphi (x) =   T_U (\sigma (x) , 2 z ) = (S^2 (\sigma (x)), 2z+2)$$
and
$$\varphi \circ T (x) =  \varphi (S(x) , z+1) = ( \sigma (S(x)) , 2(z +1)) = (S^2 (\sigma (x)) , 2z+2).$$

Hence, $T_U \circ \varphi = \varphi \circ T$, and $\varphi $ is a conjugacy and $(X_{\sigma } \times \mathbb{Z}_{3}, T)$ is 
self-induced.
\end{proof}

\subsection{Self-induced example of infinite entropy} 
\label{Example:non}
In this second example, we show that self-induced 
systems need not be uniquely ergodic, and need not be zero entropy.

The key ingredients will be that there exist Toeplitz subshifts that have a binary odometer as a factor, 
and that such systems can have positive entropy and can have multiple ergodic measures (see \cite{SWill} or \cite{DToep}).

Let $(K_0,S_0)$ be any Toeplitz subshift with the odometer 
$\mathbb{Z}_2$ as a factor.
The system $(K_0,S_0)$ is the orbit closure of a Toeplitz  sequence $z \in K_0$ that is {\em regularly recurrent} {\em i.e.},
 for any open neighborhood $U$ of $z$, there
is a $m \geq 0$ such that $S_0^{mn}(z) \in U$ for all $n \in \mathbb{Z}$. 
Because $(K_0,S_0)$ has $\mathbb{Z}_2$ as a factor, there is a sequence of 
 factor maps 
 $\pi_{n} \colon (K_{0}, S_{0}) \to ({\mathbb Z}/2^{n}{\mathbb Z}, x \mapsto x+1)$ such that 
 $\pi_{n}(z) = 0$. 

For each integer $n\ge 0$, let   $K_{n} := \pi_{n}^{-1}(0)$. We get that  
$K_0 \supset K_1 \supset K_2 \supset \cdots$  and for each $n$, 
$\{K_n, S_0 (K_n), \ldots, S_0^{2^n-1}(K_n) \}$ forms a clopen partition of $K_0$. 

For $n \geq 1$, let $S_n$ denote the induced map of $S_0$ 
on $K_n$, $S_n:K_n \to K_n$.  
So $S_n$ is equal to $S_{0}^{2^{n}}= S_{n-1}^2$ on $K_n$ for all $n \geq 1$. 

For $n <0$, we recursively define a sequence of spaces 
$K_0 \hookrightarrow K_{-1} \hookrightarrow K_{-2} \hookrightarrow \cdots$ 
and maps $S_n: K_n \to K_n$ so that $(K_{n-1},S_{n-1})$ is an 
exduced system of $(K_n,S_n)$ for the constant function 2. 
In other words, 
$K_{n-1} = K_n \times \{0,1\}$, 
$S_{n-1}:(x,0) \mapsto (x,1)$, and $S_{n-1}:(x,1) \mapsto (S_n x,0)$. 
We also observe $S_{n-1}^2 ( K_n  \times \{0\}) =  K_n \times \{0\}$ and 
$S_{n-1}^2:(x,0) \mapsto (S_n x,0)$.

Consider $\mathbf{K} = \prod_{n \in \mathbb{Z}} K_n$, and let 
$\textbf{z}$ denote the sequence $(z_n)_{n }$ where $z_n = z$ for all 
$n \geq 0$ and $z_{n-1} = (z_n,0)$ for $n \leq 0$.
Note that for all $n$, the point $z_n$ is regularly recurrent for the system
$(K_n,S_n)$. 

Define $\mathbf{T}: \textbf{K} \to \textbf{K}$, 
the map which takes the sequence $(x_n)_{n}$ to $(S_n(x_n))_{n}$. 
The map $\mathbf{T}$ is a homeomorphism 
since each $S_n$ is a homeomorphism. Consider $X \subset \textbf{K}$, the orbit closure of
$\textbf{z}$ under $\textbf{T}$.  

\begin{property}
The system $(X,\mathbf{T})$ is a minimal Cantor system. 
\end{property}

\begin{proof}
It suffices to show that $\textbf{z}$ is regularly recurrent for $\textbf{T}$.
Consider an open set $U\subset \textbf{K}$ containing $\textbf{z}$.
Without loss of generality, we
may assume $U$ is a basic set of the form 
$U= \prod U_n$ where $U_n = K_n$ for all $|n|\geq N$.

For each  $|n|<N$,  because $z_n \in K_n$ is regularly recurrent for $S_n$, there is a 
$p_n$ such that $S_n^{k p_n}(z_n) \in U_n$ for all $k \in \mathbb{Z}$. 
Let $p = \prod_{n=-N}^N p_n$. Then, for all $k \in \mathbb{Z}$, 
$\textbf{T}^{kp}(\textbf{z}) \in U$. Therefore, $\textbf{z}$ is regularly recurrent for $\textbf{T}$.
\end{proof}

\begin{property}
The system $(X,\mathbf{T})$ is self-induced.
\end{property}
\begin{proof}
Notice that there is a natural embedding $\varphi_n:K_{n+1} \hookrightarrow K_n$ for all $n \in \mathbb{Z}$: 
For $n \geq 0$, $\varphi_n$ is just inclusion, 
for $n \leq 0$, $\varphi_n(x_{n+1} ) = (x_{n+1} ,0)$, $x_{n+1} \in K_{n+1}$.
For all $n \in \mathbb{Z}$ we have the relation
$$
\varphi_n \circ  S_{n+1}=S_n^2 \circ  \varphi_n.
$$
Consider the homeomorphism $\varphi$ on the product space for which 
$\varphi(x)_n = \varphi_n(x_{n+1})$ for all integer $n$ where $x = (x_{n})_{n}$. Then, we have 
\begin{align}\label{eq:1}
\varphi \circ \textbf{T}=\textbf{T}^2 \circ \varphi.
\end{align}

We claim that $X$ is the disjoint union of $\varphi(X)$  and $\mathbf{T}(\varphi(X))$.
To prove this, observe that $\varphi(X) \cup \mathbf{T}(\varphi(X))$ is a closed and $\mathbf{T}^{2}$-invariant set by Equation \eqref{eq:1}.
So, by minimality,  we get $\varphi(X) \cup \mathbf{T}(\varphi(X) )=X$. 
Then, notice that $\varphi(X) \subset   \{(x_{n})_{n} \in X : x_0 \in K_1 \}$ and $\mathbf{T}(\varphi(X)) \subset   \{(x_{n})_{n} \in X : x_0 \in S_{0}(K_1) \}$. Hence these two sets are disjoint. 

It follows that $\varphi(X)$ is a clopen proper subset of $X$ and the induced map on $\varphi(X)$ is the map $\mathbf{T}^{2}$. Finally Equation \eqref{eq:1} enables to conclude the system $(X, \mathbf{T})$ is self-induced.
\end{proof}

From here we observe that the 
projection map from $(X,\textbf{T})$ onto the 
zero-th coordinate yields a factor map from 
$(X,\textbf{T})$ onto $(K_0,S_0)$.
By selecting $(K_0,S_0)$ in a particular way, we obtain 
self-induced examples with infinite entropy and which 
are not uniquely ergodic. 

\begin{proposition}
There exists a self-induced Cantor system with infinite entropy. 
\end{proposition}
\begin{proof}
There exists a Toeplitz subshift $(K_0,S_0)$ with a binary odometer factor which have positive entropy
\cite{DToep, SWill}. 
We may use the above construction to produce a self-induced Cantor system $(X,\textbf{T})$
with $(K_0,S_0)$ as a factor. Since self-induced systems can only have entropy 0 or $\infty$, 
and the entropy of $(X,\textbf{T})$ is greater than that of $(K_0,S_0)$, it must be infinite. 
\end{proof}

\begin{proposition}
There exists a self-induced Cantor system which is not uniquely ergodic. 
\end{proposition}
\begin{proof}
There exist Toeplitz subshifts $(K_0,S_0)$ with a binary odometer factor which have several  ergodic 
measures \cite{DToep, SWill}. 
We may use the above construction to produce a self-induced Cantor system $(X,\textbf{T})$
with  $(K_0,S_0)$ as a factor. Such a map must induce a surjection from the set of 
ergodic measures for $(X,\textbf{T})$ onto the ergodic measures for $(K_0,S_0)$.
\end{proof}


\section{Self-Induced Systems as Generalized Substitutions}\label{sec:charSelfInduced}


We wish to give a characterization of all 
self-induced minimal Cantor systems as \emph{generalized substitution subshifts}. The term generalized here refers to the fact that our substitutions may be defined on an infinite (compact zero dimensional) alphabet. In this section we define generalized substitutions and show that a minimal Cantor system is self-induced if and only if it is conjugate to a generalized substitution subshift. 


\subsection{Generalized Subshifts}


We say that a topological space $K$ is an \emph{alphabet space} if it is a compact 
zero-dimensional metric space with at least two points. 
We will refer to the elements of an alphabet space as \emph{letters}.
A helpful example to keep in mind will be one where 
$K = \overline{\mathbb{N}_0} = \{0,1,\ldots ,\infty\}$, 
the one-point compactification of $\mathbb{N}_0$. We will continue
to develop this example throughout this section.

Let $K^n$ denote the set of words of length $n$ on $K$, which we will write as $u_1u_2\ldots u_n$
as opposed to the ordered $n$-tuple $(u_1,u_2,\ldots ,u_n)$, but the topology is the same. 
Let $K^{+}=\bigcup_{n \geq 1} K^n$ denote the set of all words on $K$.  
If $u\in K^{+}$, let $|u|$ denote the length of $u$, that is, $|u|=n$ if $u$ belongs to $K^n$.
If $d_K$ is a distance for $K$, then $d_{K^+}$ will be the distance on $K^+$ defined by $d_{K^+} (u,v) = \max \{ d_K (u_i , v_i) : i\leq \min (|u| , |v| ) \}$.

Elements of the product space $K^{\mathbb{Z}}$, endowed with the product topology, are bi-infinite sequences in the alphabet space $K$. For a sequence ${\bf x}= (x_i)$ in $K^{\mathbb{Z}}$, we will use the notation ${\bf x}[i,j]$ to denote the word 
$x_i x_{i+1} \ldots x_j$ belonging to $K^{j-i+1}$.
We may consider the action of the shift map $S$ on the space of all sequences in $K^{\mathbb{Z}}$, 
$S({\bf x})_i = x_{i+1}$ for every  ${\bf x}= (x_i) \in K^{\mathbb{Z}}$.

A \emph{generalized subshift}  is a couple $(\Omega,S)$  where $\Omega$ is a closed 
$S$-invariant subset of $K^{\mathbb{Z}}$.

Note that if $X$ is a Cantor set 
then, $X^{\mathbb{Z}}$ is also a Cantor set. Moreover, 
every continuous action of a Cantor set $T:X \to X$ is topologically conjugate to a generalized 
subshift via the conjugacy $h:X \to X^{\mathbb{Z}}$ where
$$
h:x \mapsto \ldots T^{-2}xT^{-1}x.xTxT^{2}x\ldots .
$$


\subsection{Generalized Substitutions}


In order to reasonably define a substitution on an alphabet space $K$, 
we need to deal with several topological considerations that are trivial in the case where $K$ is finite (and discrete). 
For a word $w \in K^+$ and $1 \leq j \leq |w|$, let $\pi_j(w)$ denote the $j$th letter of $w$. 
We say that $\sigma:K \to K^+$ is a \emph{generalized substitution on $K$} if 
$a \mapsto |\sigma(a)|$ is continuous and 
the projection map $\pi_j \circ \sigma$ is continuous on the set $\{a \in K : |\sigma(a)| \geq j\}$. 
The words $\sigma (z)$, $z\in K$, are called $\sigma${\em -words}.

For instance, we may define a generalized substitution $\xi$ on $\overline{\mathbb{N}_0}$
by 
$$
\xi:j \mapsto 0 (j+1), \hspace{1in} \xi: \infty \mapsto 0\infty .
$$

S. Ferenczi \cite{Ferenczi:2006} observed that the system it generates is minimal, uniquely ergodic and measure theoretically conjugate to the odometer on $\mathbb{Z}_2$.
More generally, in this paper substitutions on countable alphabets are investigated.

As in the case of a finite alphabet, a generalized substitution 
$\sigma$ on an alphabet space $K$ can be extended to a function $\sigma:K^{+} \to K^{+}$ by concatenation, and therefore iterated. 
Observe that $\sigma^k$ is again a generalized substitution for any $k$.
The generalized substitution map 
$\sigma$ also extends to a function ${\sigma}:K^{\mathbb{Z}}\to K^{\mathbb{Z}}$ in a similar manner. 
For ${\bf x }\in K^{\mathbb{Z}}$, the image ${\sigma}({\bf x})$ is the sequence formed by the concatenation of 
$\sigma(x_i)$ for all $i$, with $\sigma(x_0)$ starting at the origin
$${\sigma}: \ldots x_{-2}x_{-1}.x_0x_1x_2\ldots \mapsto 
\ldots \sigma(x_{-2})\sigma(x_{-1}).\sigma(x_0) \sigma(x_1) \sigma(x_2) \ldots $$

\begin{lemma}
Let $K$ be an alphabet space and $\sigma:K \to K^+$ a generalized substitution on $K$. Then, 
${\sigma}:K^{\mathbb{Z}} \to K^{\mathbb{Z}}$ is continuous. 
\end{lemma}

\begin{proof}
Fix ${\bf x} =(x_{i}) \in K^{\mathbb{Z}}$ and let $\epsilon>0$ be given. 

Since the topology on 
$K^{\mathbb{Z}}$ is given by the product topology, there is an $\epsilon_1 > 0$ and an $N>0$ such that
given ${\bf y} = (y_i), {\bf y'} = (y_i') \in K^{\mathbb{Z}}$, $d(y_i,y_i')< \epsilon_1$ for $|i| < N$ implies $d({\bf y},{\bf y'}) < \epsilon$, where $d$ is a distance defining the topology of $K^{\mathbb{Z}}$. 

Since ${z} \mapsto |\sigma({z})|$ is a continuous function from $K$ to $\mathbb{N}$, it is uniformly continuous. 
Similarly, each $\pi_j \circ  \sigma$ is uniformly continuous. As such there is a $\delta_1 >0$ such that
for any $z,z' \in K $ with $d_K(z,z')< \delta_1$ we have $|\sigma(z)| = |\sigma(z')| \text{ and } d_{K^+} (\pi_j \circ \sigma(z),\pi_j \circ \sigma(z'))<\epsilon_1 \text{ for all } j$ with $1\leq j \leq |\sigma(z)|$.

Again by the nature of the product topology, there is a $\delta>0$ so that 
$d((x_{i}),(x'_{i}))<\delta$ implies $d_K(x_i,x_i')<\delta_1$ for $|i|<N$.
 Therefore, as $d((x_{i}),(x'_{i}))<\delta$ we deduce $d_K(x_i,x_i')<\delta_1$ for all $i$ with $|i|<N $, and, thus,
$d_K(\sigmahat(x)_i,\sigmahat(x')_i)<\epsilon_1$ for all $i$ with $|i|<N$.
Therefore, we finally obtain $d(\sigmahat((x_{i})),\sigmahat((x'_{i}))) < \epsilon $.
\end{proof}

Given a generalized substitution $\sigma$ on an alphabet space $K$, we shall consider $\mathcal{L}(\sigma)$ 
the \emph{language generated by $\sigma$}, again a trickier notion to define than in the classical case. 
Fix a letter $a$ in the alphabet space $K$. By the \emph{language generated by $a$}, 
denoted $\mathcal{L}(\sigma , a)$, we mean the set of words $w \in K^+$ such that $w$ is a subword of
$\sigma^j(a)$ for some $j \in \mathbb{N}$, or $w$ is the limit in $K^n$ of such words. 
We set $\mathcal{L}(\sigma) = \cup_a \mathcal{L}(\sigma , a)$.

Returning to our example $\xi$, one can check that the word $\infty 0$ is not a subword
of $\xi^k(j)$ for any $j \in \overline{\mathbb{N}_0}$. However, it is a limit of such 
words since $(n-1)0$ occurs as a subword of $\xi^n(0)$ for every $n \geq 2$. 
Thus, $\infty 0$ is in $\mathcal{L}(\xi , 0)$.

As in the classical case we wish to assume that our generalized substitution is \emph{primitive} in the following
sense. 

\begin{definition}
Let $\sigma : K \to K^+$ be a generalized substitution on an alphabet space $K$. We say $\sigma$ is \emph{primitive} if
given any non-empty open set $V \subset K$, there is an $j \in \mathbb{N}$ such that 
for any letter $a \in K$ and any $k \geq j$, one of the letters of $\sigma^k(a)$ is in the set $V$.
\end{definition}

With the assumption of primitivity and that $\# K>1$, we see that $|\sigma^n(a)| \to \infty$ for any $a \in K$. 

One can check that our example $\xi$ is primitive (recall that in this context open sets are complementary of finite sets in $\overline{\mathbb{N}_0}$). 

\begin{proposition}
Let $K$ be an alphabet space and let $\sigma : K \to K^+$ be a primitive generalized substitution.
Given any $a \in K$, $\lim_{n \to \infty} |\sigma^n(a)|=\infty$. 
\end{proposition}
\begin{proof}
By our axioms, $K$ contains at least two nonequal letters $a,b$. 
Let $U, V$ be two non-intersecting open sets containing $a$ and $b$, respectively. 
By primitivity, for any $x \in K$, there is a $k$ such that 
$\sigma^k(x)$ contains a point in $U$ and a point in $V$, so in particular $|\sigma^k(x)|\geq 2$. 
It follows from the continuity of the functions $\pi_j\circ \sigma$ and $x \mapsto |\sigma(x)|$, and, the compactness of $K$ that 
there is a $j$ such that $|\sigma^j(x)|\geq 2$ for all $x \in K$. 

Inductively, we see that $|\sigma^{jn}(x) | \geq 2^n$ for all $x \in K$, completing the proof. 
\end{proof}

The assumption of primitivity also simplifies the definition of the language of a generalized substitution
to be the language generated by $a$ for any $a \in K$: $\mathcal{L} (\sigma ) = \mathcal{L} (\sigma , a)$. 
These are all the same as the following
proposition shows. 

\begin{proposition}
Suppose $\sigma:K \to K^+$ is a primitive generalized substitution on an alphabet space $K$. 
Then, for any two letters $a, b \in K$, $\mathcal{L}(\sigma , a) = \mathcal{L}(\sigma , b)=\mathcal{L} (\sigma )$. 
\end{proposition}

\begin{proof}
Let $w \in \mathcal{L}(\sigma , b)$ and suppose $\epsilon>0$ is given. Then, there are integers $k,i,j \geq 0$ such that
the distance between $\sigma^k(b)[i,j]$ and $w$ is less than $\epsilon/2$ in the $K^{|w|}$-metric.
By the continuity condition on $\sigma^k$, there is a $\delta>0$ such that 
$d_{K}(b,b')<\delta$ implies the distance from $\sigma^k(b)[i,j]$ and $\sigma^k(b')[i,j]$ is less than $\epsilon/2$.

Let $U$ be the ball of radius $\delta$ around $b$ in $K$. By the primitivity of $\sigma$, there is an integer 
$n$ and a $b' \in U$ such that $b'$ occurs in $\sigma^n(a)$. 
It follows that there are integers $i', j' \geq 0$ such that 
the distance from $\sigma^{k+n}(a)[i',j']$ to $w$ is less than $\epsilon$. 
Thus, $w $ belongs to $ \mathcal{L}(\sigma , a)$ and
$\mathcal{L}(\sigma , b)$ is included in  $\mathcal{L}(\sigma , a)$. 

Similarly, $\mathcal{L}(\sigma , a) \subset \mathcal{L}(\sigma , b)$. 
\end{proof}

For a primitive generalized substitution $\sigma : K \to K^+$,
define 
$X_{\sigma} \subset K^{\mathbb{Z}}$ to be set of sequences ${\bf x} \in K^{\mathbb{Z}}$ such that
${\bf x}[-n,n] \in \mathcal{L}(\sigma)$ for all $n \geq 0$.
It follows that $X_{\sigma}$ is a generalized subshift, {\em i.e.}, a closed, $S$-invariant subset of $K^{\mathbb{Z}}$.

When $\sigma$ is a primitive substitution, we need to explicitly rule out the possibility that 
$X_{\sigma}$ contains periodic points, an assumption that again is 
necessary even in the classical case. 
In such a situation we will say that $\sigma$ is \emph{aperiodic}.

Our example $\xi $ being primitive, its language $\mathcal{L} (\xi )$ contains infinitely many letters and thus is aperiodic.

Because we wish to focus on self-induced minimal Cantor systems 
we shall require additional properties for our generalized
substitutions, that they be primitive and aperiodic. Below, we establish an 
equivalent description of $X_{\sigma}$ under these assumptions. This equivalent description is 
more or less the analogous technique to finding a fixed point 
for $\sigmahat$ in the classical substitution case and taking its orbit closure. 

For ${\bf x} \in K^{\mathbb{Z}} $ we denote by $\omega_\sigma ({\bf x})$ the {\em omega-limit set} of ${\bf x}$ under the map 
$\sigmahat$: $\omega_\sigma ({\bf x}) = \bigcap_{n \in \mathbb{N}} \overline{ \left\{ \sigmahat^m({\bf x}) : m \geq n \right\}}$.

\begin{proposition}
\label{prop:minimality}
Suppose $K$ is an alphabet space and $\sigma:K \to K^+$ is a primitive
generalized substitution, and let $a \in K$. Suppose $bc$ appears in $\sigma^j(a)$ for some $j >0$
and let ${\bf x}$ be any point in $K^{\mathbb{Z}}$ with $x_{-1}=b\in K$ and $x_0=c\in K$.
Then, for ${\bf z} \in \omega_\sigma ({\bf x})$, $X_{\sigma}$ is equal to the closure of the $S$-orbit of ${\bf z}$. 
\end{proposition}
\begin{proof}
Suppose ${\bf z} \in \omega_\sigma ({\bf x})$. Then, ${\bf z} = \lim_{i \to \infty} \sigmahat^{k_i} ({\bf x})$. 
Fix $m \in \mathbb{N}$ and let $\epsilon >0$ be given. Then, there is an $i \geq 1$ such that 
the word ${\bf  z}[-m,m]$ is within $\epsilon$ of a subword of $\sigma^{k_i}(bc)$, which implies ${\bf z}[-m,m]$ is 
within $\epsilon$ of a subword of $\sigma^{k_i+j}(a)$. Therefore, ${\bf z}[-m,m] \in \mathcal{L}(\sigma , a) = \mathcal{L}(\sigma)$. 
This shows that ${\bf z}$ belongs to $X_{\sigma}$.
But because $X_{\sigma}$ is closed and shift-invariant, 
we see that $\overline{\{S^n ({\bf  z}) : n\in \mathbb{Z} \}}$ is included in  $X_{\sigma}$.

Conversely, fix an ${\bf y}\in X_{\sigma}$ and an  $\epsilon'>0$.  Let us fix an integer  $n \ge 0$, such that   $d({\bf y},{\bf y'})< \epsilon'$ when  $d_{K^{+}}( {\bf y}[-n,n], {\bf y'}[-n,n])$ is less than $\epsilon'/2$. 
Since ${\bf y}[-n,n] \in \mathcal{L}(\sigma)= \mathcal{L}( \sigma, c)$, there are integers $k,i$ such that 
$d_{K^{+}}(\sigma^k(c)[i-n,i+n],{\bf y}[-n,n]) $ is less than $\epsilon'/4$. 

Note that there is a $\delta>0$ such that if $d_{K}(c,c')<\delta$ 
then, the words $\sigma^k(c)$ and $\sigma^k(c')$ are within $\epsilon'/4$. 
Set $U$ equal to the $d_K$-ball of radius $\delta$ around $c$. 
By the primitivity of $\sigma$, we can find a point in $U$ in the word $\sigma^l(c)$ for all $l$ large enough. 
Thus, there is an integer $i(l)$ such that $d_{K^{+}}(\sigma^{l+k}(c)[i(l)-n,i(l)+n], {\bf y}[-n,n]) $ is less than $\epsilon'/2$. 
Since $\epsilon'$ and $l$ are arbitrary,  it follows that ${\bf y} \in \overline{ \{ S^n ({\bf z}) : n \in \mathbb{Z} \} }$.
\end{proof}

Once again returning to our example $\xi$, letting ${\bf x}$ be any sequence in $\overline{\mathbb{N}_0}$ with $x_{-1}=0$ and $x_0=1$,
we see that the omega-limit set of ${\bf x}$ is simply the point  
$${\bf z}= \ldots  0 1 0 2 0 1 0 \infty . 0 1 0 2 0 1 0 3 0 1 \ldots $$
Note ${\bf z}$ is a fixed point for $\widehat{\xi}$. 
The set $X_{\xi}$ is the $S$-orbit closure of this point. 

\begin{proposition}\label{prop:concatenation}
Suppose $K$ is an alphabet space and $\sigma:K \to K^+$ is a primitive
generalized substitution.
For ${\bf z} \in X_{\sigma}$ and any $k \in \mathbb{N}$, the sequence ${\bf z}$ is a concatenation 
of $\sigma^k$-words and ${\bf z} = S^i \sigma^k ({\bf y})$ for some $i\geq 0$ and ${\bf y}\in X_\sigma$. 
\end{proposition}
\begin{proof}
First let ${\bf z} \in \omega_\sigma ({\bf x})$ for some ${\bf x} \in K^{\mathbb{Z}}$. 
Then, ${\bf z} = \lim_{k \to \infty} \sigmahat^{n_k}({\bf x})$. 
From Proposition \ref{prop:minimality} one can suppose ${\bf x}$ belongs to $X_\sigma$.
Consider 
the sequence $(\sigmahat^{n_k-1}({\bf x}) )$. 
By compactness a subsequence converges, say to a point ${\bf y}$ belonging to $X_\sigma$. By 
continuity of $\sigmahat$, it follows that $\sigmahat({\bf y})={\bf z}$. A similar argument shows that for any $k \in \mathbb{N}$, ${\bf z} = \sigma^k ({\bf y}')$, for some ${\bf y}'\in X_\sigma$. 

 By continuity and compactness, the maximal length of $\sigma^{k}$-words is bounded. 
Furthermore  observe that $\bigcup_{0 \le i  <\max_{a\in K} |\sigma^{k}(a)|} S^{i} \sigma^{k}(X_\sigma ) $ is a closed $S$-invariant set. By Proposition \ref{prop:minimality} and the first part of the proof, this set also contains  $X_{\sigma}$. This proves the proposition.\end{proof}

This leaves us with the issue of \emph{recognizability} in $X_{\sigma}$, namely the uniqueness of 
the decomposition of a sequence in $X_{\sigma}$ into $\sigma^k$-words. 
It was proven by Moss\'e that a condition equivalent to our definition of recognizability below follows from primitivity and aperiodicity 
in the case of a classical substitution \cite{Mosse:1992, Mosse:1996}. 
We do not yet know if this theorem holds in the case
of generalized substitutions, so we will assume recognizability as a separate axiom. 

\begin{definition}
Let $\sigma : K \to K^+$ be a generalized substitution on an alphabet space $K$. We say $\sigma$ is \emph{recognizable} if
for every ${\bf z} \in X_{\sigma}$, there is a unique set of integers $\{n_k : k \in \mathbb{Z} \}$ and unique
${\bf x} \in X_{\sigma}$ such that $\sigma(x_k)={\bf z}[n_k,n_{k+1}-1]$ for all $k \in \mathbb{Z}$. 
\end{definition}

The generalized substitution $\xi $ in our example is recognizable. 
For ${\bf z} \in X_{\xi}$, since $0$ appears in the image of each letter exclusively at the initial position, so either we have $(n_k)=(2k)$ or $(n_k)=(2k+1)$. 
Moreover, $\xi$ being injective on the letters,
it follows immediately that $\xi$ is recognizable.

\begin{proposition}
\label{prop:recsubgen}
Suppose $\sigma$ is recognizable and suppose ${\bf z} = \sigmahat({\bf x})$ with ${\bf x} \in X_{\sigma}$ . 
Then, the first return time to $\sigmahat(X_{\sigma})$ of ${\bf z}$ with respect to the shift map is $|\sigma(x_0)|$. 
\end{proposition}

\begin{proof}
Let $r = |\sigma(x_0)|$.
Then, $S^r({\bf z}) = \sigmahat(S({\bf x})) \in \sigmahat(X_{\sigma})$. 
However, if $S^j({\bf z}) = \sigmahat({\bf y})$ for $0 < j < r$ where ${\bf y} \in X_{\sigma}$
then, this violates the uniqueness conditions in recognizability.
\end{proof}


\subsection{Self-Induced Systems and Generalized Substitutions}


With all of these notions established for generalized substitutions, we are prepared to proceed
in showing that a minimal Cantor system is self-induced if and only if it is conjugate to 
a recognizable, primitive, aperiodic generalized substitution subshift.

\begin{theorem} \label{gen1}
Let $(X_{\sigma},S)$ be generated by a  primitive generalized substitution $\sigma: K \to K^+$ then:
\begin{enumerate}
\item
\label{enum:min}
$(X_{\sigma},S)$ is minimal;
\item
\label{enum:rec}
If $\sigma $ is recognizable and aperiodic then, $(X_{\sigma},S)$ is a self-induced minimal Cantor system, where $\sigmahat$ is a conjugacy from $(X_{\sigma} , S)$  to the induced system $(\sigmahat (X_{\sigma}) , S_{\sigmahat (X_{\sigma}) })$.
\end{enumerate}
\end{theorem}
\begin{proof}
\eqref{enum:min}
\textbf{($X_{\sigma}, S)$ is a minimal Cantor system.}
Let ${\bf x}\in K^\mathbb{Z}$ where $x_{-1}x_0$ is a subword of $\sigma^l(a)$ for a fixed $a \in K$ and $l>0$. 
Let ${\bf z} \in \omega_\sigma ({\bf x})$. 
Thus, from Proposition \ref{prop:minimality}, the $S$-orbit closure of ${\bf z}$ is $X_\sigma$.
To show the minimality it is enough to prove that for every open set $U$ in $X_{\sigma}$ containing ${\bf z}$,
there is an $R>0$ such that for any $j \in \mathbb{Z}$, there is a $0 \leq i < R$ such that $S^{j+i}({\bf z}) \in U$. 
Notice that there are an $n \in \mathbb{N}$ and an $\epsilon>0$ such that $S^i({\bf z})$ in $U$ whenever
the distance from ${\bf z}[-n,n]$ to ${\bf z}[i-n,i+n]$ is less than $\epsilon$. 

Moreover, from the continuity of $\sigma$ and the very definition of ${\bf z}$, there are a $\delta>0$ and a $k>0$ such that $d_K(a,a')<\delta$ implies that a word 
appears in $\sigma^k(a')$ which is within $\epsilon$ of ${\bf z}[-n,n]$. 

By the primitivity of $\sigma$, there is an $m$ such that for any $b \in K$, 
$\sigma^m(b)$ contains a letter within $\delta$ of $a$. Accordingly, 
$\sigma^{m+k}(b)$ contains a word within $\epsilon$ of ${\bf z}[-n,n]$. 

 Let $R$ be equal to $\sup_{b\in K} |\sigma^{k+m}(b)|$. 
 It is finite from the compactness and the continuity assumption. 
 The sequence ${\bf z}$ being a concatenation of $\sigma^{m+k}$-words (Proposition \ref{prop:concatenation}),
for any $j \in \mathbb{Z}$, there is a $0 \leq i < R$ such that $S^{j+i}({\bf z}) \in U$.

\eqref{enum:rec}
\textbf{$(X_{\sigma}, S)$ is self-induced.}
Because the system is aperiodic, it is a minimal Cantor system. 
The map $\sigmahat: X_{\sigma} \to X_{\sigma}$ is  continuous by assumption, thus, the image set $U = \sigmahat(X_{\sigma})$ is compact (closed) in $X_{\sigma}$.
Furthermore, it is injective 
because $\sigma$ is recognizable. 
Consequently, every point in $X_{\sigma}$ can be decomposed into 
$\sigma$-words and every point in $X_{\sigma}$ is in $S^i(U)$ for some $i \in \mathbb{Z}$. 

Set $X_i = \{ {\bf x} \in X_{\sigma} : |\sigma(x_0)|=i\}$ and $U_i = \sigmahat(X_i)$. 
Since the $X_i$'s are disjoint and $\sigmahat$ is injective, so are the $U_i$'s.
Proposition \ref{prop:recsubgen} ensures 
the return time to $U$ on $U_i$ is $i$.
So we have that $\{ S^k (U_i) : 0 \leq k <i, 1\leq i \leq \max_{b\in K} |\sigma (b)| \}$
forms a disjoint collection of sets which covers $X_{\sigma}$. 
Thus, each 
$S^k(U_i)$ is clopen and $U$ is clopen. 

Let ${\bf x} \in X_{\sigma}$, and set $r=|\sigma(x_0)|$. 
Then, again from Proposition \ref{prop:recsubgen},
$r$ is the return time of $\sigmahat({\bf x})$ to the set $U$ and  $S^r(\sigmahat({\bf x}))=\sigmahat(S({\bf x})) \in U$. 
\end{proof}

We now establish the reverse direction, that is to show that any 
self-induced minimal Cantor system is conjugate to a system generated by 
a primitive, recognizable, aperiodic generalized substitution.

\begin{theorem} \label{gen2}
Suppose $(X,T)$ is a self-induced minimal Cantor system. 
Then, $(X,T)$ is conjugate to a recognizable, primitive, aperiodic, generalized substitution subshift $(X_{\sigma},S)$.
\end{theorem}
\begin{proof}
Suppose $U \subset X$ is clopen and $(X,T)$ is 
self-induced via a conjugacy $\varphi:(X,T) \to (U,T_U)$. 
Recall that in Proposition \ref{wlog} we showed that without loss of generality 
we may assume $U \cap TU = \emptyset$. 

Set $K=X$, define $\sigma:K \to K^+$ by 
$$
\sigma(x) = \varphi(x) T(\varphi(x)) T^2(\varphi(x)) \cdots T^{r_U(\varphi ( x))-1}(\varphi(x)),
$$ 
where $r_{U}(\varphi(x))$ is the return time of $\varphi(x)$ to $U=\varphi (X)$. 

\medskip

\textbf{$\sigma$ is a generalized substitution.}
To see that $m:x \mapsto |\sigma(x)|$ is continuous, we simply note that $|\sigma(x)|=r_U(\varphi(x))$ and that $r_{U}\circ \varphi$ is continuous. 
That $\pi_j \circ \sigma$ is continuous follows from observing
$\pi_j \circ \sigma(x) = T^{j-1} \circ \varphi(x)$ if $1\le j \le  r_{U} (\varphi(x))$.

\medskip

\textbf{$\sigma$ is primitive.}
We will show by induction on $n\geq 1$ that, for each $x\in X$,  
\begin{align}\label{eq:SubstFormula}
\sigma^n(x) = \varphi^n(x) T(\varphi^n(x)) T^2(\varphi^n(x)) \cdots T^{r_{\varphi^{n}(X)}(\varphi^n(x))-1}(\varphi^n(x)).
 \end{align}
 
Suppose it is true for $n$.
For $0\leq i \leq r_{\varphi^{n}(X)}(\varphi^n(x))-1$ we have 
$$
\sigma( T^i \varphi^n (x)) = (\varphi T^i \varphi^n (x)) (T\varphi (T^i \varphi^n (x)))  \cdots T^{r_U (\varphi (T^i \varphi^n (x))) -1} \varphi (T^i \varphi^n (x)).
$$
Since $\varphi$ is a conjugacy map between $(X,T)$ and $(U,T_{U})$, we get $\varphi(T^{i}\varphi^n (x)) = T_{U}^{i}\varphi^{n+1}(x)$, so that for $0\leq j \leq  r_{U}(\varphi T^{i} \varphi^n(x))-1 = r_U (T_U^i  \varphi^{n+1} (x)) -1$, we have 
\begin{align*}
T^j \varphi (T^i \varphi^n (x)) & = T^j T_U^i \varphi^{n+1} (x) \hbox{ and }\\
T^{r_U (T_U^i  ( \varphi^{n+1} (x))) -1} \varphi (T^i \varphi^n (x)) & = T^{-1} T_U^{i+1} \varphi^{n+1} (x).
\end{align*}

Then, to finish the proof of the induction,  it remains to show that the last letter of $\sigma( T^{r_{\varphi^n(x)}(\varphi^n(x)) -1} \varphi^n (x))$ is the correct one, that is
$$
T^{r_{\varphi^{n+1}(X)}(\varphi^{n+1}(x))-1} \varphi^{n+1} (x) = T^{-1} T_U^{r_{\varphi^{n}(X)}(\varphi^n(x))} \varphi^{n+1} (x).
$$
This last equality  actually comes from
\begin{align*}
T_U^{r_{\varphi^{n}(X)}(\varphi^n(x))} \varphi^{n+1} (x) & = \varphi (T^{r_{\varphi^{n}(X)}(\varphi^n(x))} \varphi^{n} (x) )\\
& = \varphi (T_{\varphi^{n}(X)} \varphi^{n} (x))\\
& =  T_{\varphi^{n+1}(X)} \varphi^{n+1} (x),
\end{align*}
where the last equality is provided by Proposition \ref{prop:recuSelfInduce}.

Since $(X,T)$ is minimal, for any open set $V \subset X$  there is an $R>0$ such that for any point 
$z \in X$, $\{z,T(z), \ldots, T^R(z)\} \cap V \neq \emptyset$. 
Thus, primitivity will follow   from  Formula \eqref{eq:SubstFormula} by showing that the return time to $\varphi^{n}(U)$
must go to infinity with $n$.

Because $U \cap TU = \emptyset$, for all $x \in X$ the $T$-return time 
for $x$  to $U$ is  at least $2$. Similarly, the $T_U$-return time for a point $x \in \varphi(U)$ 
to $\varphi(U)$ is at least $2$, which implies the $T$-return time 
for a point $x \in \varphi(U)$ to $\varphi(U)$ is at least $4$. Accordingly, 
the $T$-return time to $\varphi^{n}(U)$ is at least $2^n$, and primitivity follows. 



\medskip

\textbf{$\sigma$ is recognizable.} Recall that any word $w$ in the langage $\mathcal{L}(\sigma)$ is a limit of subwords of words $\sigma^n(a_{n})$, that is, by Formula \eqref{eq:SubstFormula}, $w$ is a limit of words $ x_{n}Tx_{n} \ldots T^{|w|-1}x_{n}$ for some  points $x_{n} \in X$. The compactness of the space together with the continuity of the map $T$  ensure the word $w$ is of the form $ xTx \cdots T^{|w|-1}x$ for a point $x \in X$.   
Let $ \bf z$ be any sequence in $X_{\sigma}$. So any word ${\bf z}[-n,n]$, $n \ge 0$  is in $\mathcal{L}(\sigma)$, thus by  the former remark, and applying a classical diagonal extraction, we obtain that ${\bf z}$ is of the form ${\bf  z}=  \ldots T^{-2}x T^{-1}x. x Tx T^{2} x \ldots$ for some point $x \in X$.
Moreover notice that this $x$ is unique.

  Let $n_0\ge 0$ be the smallest non negative index $j$
such that $z_{j}=T^{j}x \in U$. By the surjectivity of $\varphi$, there is a $y \in X$ such that $z_{n_{0}} = \varphi(y)$. It follows, from the definitions of $\sigma$ and $\varphi$, that $S^{n_{0}}{\bf z} = \sigma(  \ldots T^{-2}y T^{-1}y. y Ty T^{2} y \ldots)$.

Inductively define $n_0 < n_1 < n_2 < \cdots$ and 
$n_0 > n_{-1} > n_{-2}> \cdots$ so that $z_j \in U$ if and only if $j=n_k$ for some $k$. 
Then, for all $k$,  ${\bf z}[n_k, n_{k+1}) = \sigma(T^{k}y)$. 

To prove the recognizability, observe that a $\sigma$-word has a letter in $U$ if and only if it is its first letter. Consequently, the sequence $(n_{k})$ is uniquely defined. Since the substitution $\sigma$ is injective on the letters, we get   $\sigma$ is recognizable.
\medskip

\textbf{$(X_{\sigma},S)$ is conjugate to $(X,T)$.}


We wish to show that there is a homeomorphism 
$h: X \to X_{\sigma}$ for which $h\circ T=S \circ h$. 
Note that the map $h:X \to X^{\mathbb{Z}}$ given by 
$$h: x \mapsto \ldots T^{-2}(x) T^{-1}(x) . x T(x) T^2(x)\cdots $$ 
gives a 
conjugacy from $(X,T)$ onto its image with the shift map. 
It is left to show that $h(X) = X_{\sigma}$. 
The substitution $\sigma$ being recognizable (see above), the system $(X_\sigma , S)$ is minimal from Theorem \ref{gen1}.
Consequently, it suffices to show that $h(X)$ is a subset of $X_{\sigma}$.

To that end, let $x \in X$ and consider the sequence 
$h(x)$. Let $I$ be the infinite set $\{i \in \mathbb{Z} : T^i(x) \in U\}$. Since $\varphi$ is bijective, for each $i\in I$,  there is a unique $y_{i} \in X$ such that 
$T^{i}(x) = \varphi(y_{i})$.  Notice that 
\begin{align*}
T^{i}(x) T^{i+1}(x) \cdots T^{i+r_U( T^{i}(x) )-1}(x) & = \varphi(y_i) T(\varphi(y_i)) \cdots T^{r_U (\varphi(y_i))-1}(\varphi(y_i)) \\
& = \sigma(y_i).
\end{align*}
Thus, the sequence $h(x)$ is a concatenation of $\sigma$-words $\sigma(y_{i})$. 

%
%
By applying the same arguments to the generalized substitution $\sigma^k$ for each integer  $k \ge 1$, associated with the set  $\varphi^{k-1}(U)$ and the  return times to that set (see  Formula \eqref{eq:SubstFormula}), we obtain  that $h(x)$ is a concatenation of $\sigma^k$-words for all $k$. Therefore, $h(x) \in X_{\sigma}$. 
\end{proof}


\section{Poincar\'e sections}
\label{section:poincare}

We end with a note about the situation where we expand our notion 
of induced maps to return maps on closed, but not 
necessarily clopen, sets. 
Let $(X,\mathcal{B} , \mu , T)$ be a measure theoretical dynamical system.
Due to Poincar\'e Recurrence Theorem (see \cite{Petersen:1983} for example), once we have $U\in \mathcal{B}$ with $\mu (U)> 0$ then, one can induce on $U$ to define the induced dynamical system $(U, \mathcal{B}\cap U , \mu_U , T_U )$, where $T_U (x)
$ is the first return of $x$ in $U$.
But it can happen that the induced map on  a Borel set $U$ is well defined even if $\mu (U)=0$. 
For example, consider the full-shift $( \{ 0,1\}^\mathbb{Z} , S)$  and $U$ the set of sequences $\dots w_{-1}.w_0 w_1 \dots $ where the $w_{2i}w_{2i+1}$ belongs to $\{ 00 , 11\}$ for all $i$. 
Observe the return time to $U$ of each element of $U$ is well-defined and equals $2$. 
Then, the induced map $S_U : U\to U$ is an homeomorphism and $(X  , S)$ is conjugate to $(U  , S_U)$. 
This suggests the following definition.

\begin{definition}
Let $(X,T)$ be a minimal Cantor system. 
We say that a nonempty closed set $C$ is a 
{\em Poincar\'e section} if the induced map $T_C : C\to C$ is a well-defined homeomorphism. 
We say that $(X,T)$ is {\em weakly self-induced} if there exists a Poincar\'e section $C$ such that 
$(X,T)$ is conjugate to $(C, T_C )$.
\end{definition}

The main result in this section is the following.
\begin{theorem}
\label{theo:weakkakutanieq}
Let $(X,T)$ and $(Y,R)$ be two minimal Cantor systems.
There exists a Poincar\'e section $C$ in $(Y,R)$ such that $(C , {R}_{C})$ is conjugate to $(X,T)$.
\end{theorem}

Before proving it, we need a technical lemma on embedding of ordered graphs. 
If $G_{1} =(V_{1},E_{1})$  $G_{2}=(V_{2},E_{2})$ are two graphs with respective orders $\preceq_{1}$, $\preceq_{2}$ on the edges, we say that $(G_{1}, \preceq_{1})$ and $(G_{2}, \preceq_{2})$ are {\em isomorphic} if there is a graph isomorphism $i \colon G_{1} \to G_{2}$ that is order preserving {\em i.e.}, for all $e,f \in E_{1}$ 
$$e\preceq_{1}f \textrm{ if and only if  } i(e) \preceq_{2} i(f).$$

We refer to Section \ref{subsec:bvrep} for the definitions involve in the notion of Bratteli diagrams.

\begin{lemma}\label{lem:tech1}
Let $B = ((V_{n}), (E_{n}), (\preceq_{n}))$ be a  simple ordered  Bratteli diagram such that the sequence $(\# V_{n})$ is increasing.  Let $V$ be a subset of  $V_{n_{0}}$ for some level $n_{0}$ and let $K = (V, V', E)$ be a bipartite graph with a RL-order $\preceq_{K}$ on $E$.  
Then there exists an integer $k$    such that  $(V_{n_{0}}, V_{n_{0}+k}, E_{n_{0}+1, n_{0}+k})$ contains a subgraph $(V, V", E')$ isomorphic to $(K, \preceq_{K})$ for the order $\preceq_{n_{0}+1,n_{0}+k}$ restricted to $E'$. 
\end{lemma}
\begin{proof} For a fixed set $V \subset V_{n_{0}}$, we show by induction on $n \ge 1$ that for every  ordered bipartite graph $K = (V, V', E)$ with $n$ edges  and with order $\preceq_{K}$, there exist an integer $k$   and a subgraph  $ (V, V", E')$ of  $(V_{n_{0}}, V_{n_{0}+k}, E_{n_{0}+1, n_{0}+k})$ such that  $( (V, V", E'), \preceq_{n_{0}+1, n_{0}+k})$ is  isomorphic to $(K, \preceq_{K})$.

The result is obvious if $n =1$.  Assume that $n > 1$ and  that the result is true for $n-1$.  Let $((V, V', E), \preceq_{K})$ be an ordered  bipartite graph with $n$ edges and     let $e= (x_{e}, y_{e})   \in E$ be a maximal element for the order $\preceq_{K}$. 
Without loss of generality, we may assume thanks the induction hypothesis, that  $(V, V', E\setminus \{e\})$ is a subgraph of $(V_{n_{0}}, V_{n_{0}+k}, E_{n_{0}+1, n_{0}+k})$ for some integer $k$, and the order $\preceq_K$ coincides with $\preceq_{n_{0}+1, n_{0}+k}$ on $E\setminus \{e\}$. Since the diagram $B$ is simple, there is an injection $i \colon V' \to  V_{n_{0} +k+2}$. By concatenating any edge  $(x,y)$ of $E\setminus \{e\}$ with the  minimal path from $y$ to $i(y)$, we obtain a subgraph of   $(V_{n_{0}}, V_{n_{0}+k+2}, E_{n_{0}+1, n_{0}+k+2})$ isomorphic to  $((V, V', E\setminus \{e\}), \preceq_{K})$ for the restriction of $\preceq_{n_{0}+1, n_{0}+k+2}$. Then we identify the edge $e$ with a maximal edge from $x_{e} \in V$ to $i(y_{e}) \in V_{n_{0}+k+2}$ in $E_{n_{0}+1, n_{0}+k+2}$. Since they are at least two paths from each vertex of $V_{n_{0}+k}$ to each vertex of $V_{n_{0}+k+2}$, the maximal paths differ from the minimal ones. Hence we obtain a subgraph of $(V_{n_{0}}, V_{n_{0}+k+2}, E_{n_{0}+1, n_{0}+k+2})$ isomorphic to $((V, V', E), \preceq_{K})$. 
\end{proof}

The strategy to prove Theorem \ref{theo:weakkakutanieq} is the following: we start fixing a Bratteli diagram  $B$ representing  $(X,T)$. Then by microscoping and contracting a Bratteli diagram associated to $(Y,R)$, we obtain one that contains a copy $\tilde{B}$ of $B$. The Bratteli-Vershik system associated to $\tilde{B}$ is then conjugated to $(X,T)$ and an induced system of $(Y,R)$.

\begin{proof}[Proof of Theorem \ref{theo:weakkakutanieq}]
Take a Bratteli-Vershik representation $(X_B , V_B)$  of   $(X,T)$    given by the ordered Bratteli diagram $B = ((V_{n}), (E_{n}), (\preceq_{n}))$. 
Contracting and microscoping sufficiently a  Bratteli diagram representing $(Y,R)$, we  obtain an ordered Bratteli diagram $B' = ((V_{n}'),(E_{n}') , (\preceq_{n}'))$  where  $(\# V'_{n})$ is increasing and for each level $n$, $\# V'_{n} \ge \# V_{n}$.
We inductively use Lemma \ref{lem:tech1} to get a contraction $B''= ((V''_{n}),(E''_{n}) , (\preceq_{n}''))$ of $B'$ such that, for each level $n$,  the bipartite ordered graph   $((V_{n}, V_{n+1}, E_{n+1}), \preceq_{n+1})$   is isomorphic to $((\tilde{V}_{n}, \tilde{V}_{n+1}, \tilde{E}_{n+1}), \preceq_{n+1}'')$ a subgraph of  $(V''_{n}, V''_{n+1}, E''_{n+1})$. 
It follows that $\tilde{B} = ((\tilde{V}_{n}), (\tilde{E}_{n}), (\preceq_{n}''))$,  is a subdiagram of $B''$  and the associated Bratteli-Vershik system $(X_{\tilde{B}} , V_{\tilde{B}})$, with $X_{\tilde{B}} \subset X_{B''}$,   is a new representation of $(X,T)$. 
Observe $X_{\tilde{B}}$ is a non empty closed set being a decreasing intersection of closed sets.
Since the induced orders coincide, the induced system of $(X_{B"}, V_{B"})$  on $X_{\tilde{B}}$ is conjugate to $(X,T)$.
 \end{proof}

\begin{corollary}
\label{theo:weakself}
Every minimal Cantor system is weakly self-induced. 
\label{thm:weak}
\end{corollary}

\section*{Acknowledgements}
We would like to thank Gabriel Vigny who found the example at the beginning of Section \ref{section:poincare}.

\bibliography{DPO}

\begin{thebibliography}{10}

\bibitem{Arnoux:1988}
P.~Arnoux.
\newblock Un exemple de semi-conjugaison entre un \'echange d'intervalles et
  une translation sur le tore.
\newblock {\em Bull. Soc. Math. France}, 116:489--500 (1989), 1988.

\bibitem{Arnoux&Ito:2001}
P.~Arnoux and S.~Ito.
\newblock Pisot substitutions and {R}auzy fractals.
\newblock {\em Bull. Belg. Math. Soc. Simon Stevin}, 8:181--207, 2001.
\newblock Journ{\'e}es Montoises d'Informatique Th{\'e}orique
  (Marne-la-Vall{\'e}e, 2000).

\bibitem{Arnoux&Rauzy:1991}
P.~Arnoux and G.~Rauzy.
\newblock Repr\'esentation g\'eom\'etrique de suites de complexit\'e $2n+1$.
\newblock {\em Bull. Soc. Math. France}, 119:199--215, 1991.

\bibitem{Boshernitzan&Carroll:1997}
M.~Boshernitzan and C.~R. Carroll.
\newblock An extension of lagrange's theorem to interval exchange
  transformations over quadratic fields.
\newblock {\em J. Anal. Math.}, 72:21--44, 1997.

\bibitem{Bressaud&Jullian:2012}
X.~Bressaud and Y.~Jullian.
\newblock Interval exchange transformation extension of a substitution
  dynamical system.
\newblock {\em Confluentes Math.}, 4:1250005, 54, 2012.

\bibitem{Dartnell&Durand&Maass:2000}
P.~Dartnell, F.~Durand, and A.~Maass.
\newblock Orbit equivalence and kakutani equivalence with sturmian subshifts.
\newblock {\em Studia Math.}, 142:25--45, 2000.

\bibitem{Dekking:1978}
F.~M. Dekking.
\newblock The spectrum of dynamical systems arising from substitutions of
  constant length.
\newblock {\em Z. Wahrscheinlichkeitstheorie und Verw. Gebiete}, 41:221--239,
  1978.

\bibitem{delJunco&Rudolph&Weiss:2009}
A.~del Junco, D.~J. Rudolph, and B.~Weiss.
\newblock Measured topological orbit and {K}akutani equivalence.
\newblock {\em Discrete Contin. Dyn. Syst. Ser. S}, 2:221--238, 2009.

\bibitem{DToep}
T.~Downarowicz.
\newblock Survey of odometers and {T}oeplitz flows.
\newblock In {\em Algebraic and topological dynamics}, volume 385 of {\em
  Contemp. Math.}, pages 7--37. Amer. Math. Soc., Providence, RI, 2005.

\bibitem{Downarowicz&Maass:2008}
T.~Downarowicz and A.~Maass.
\newblock Finite-rank {B}ratteli-{V}ershik diagrams are expansive.
\newblock {\em Ergodic Theory Dynam. Systems}, 28:739--747, 2008.

\bibitem{Durand:thesis}
F.~Durand.
\newblock {\em Contributions {\`a} l'{\'e}tude des suites et syst{\`e}mes
  dynamiques substitutifs}.
\newblock PhD thesis, Universit{\'e} de la M{\'e}diterran{\'e}e (Aix-Marseille
  II), 1996.

\bibitem{Durand:1998}
F.~Durand.
\newblock A characterization of substitutive sequences using return words.
\newblock {\em Discrete Math.}, 179:89--101, 1998.

\bibitem{Durand:2010}
F.~Durand.
\newblock Combinatorics on {B}ratteli diagrams and dynamical systems.
\newblock In {\em Combinatorics, automata and number theory}, volume 135 of
  {\em Encyclopedia Math. Appl.}, pages 324--372. Cambridge Univ. Press,
  Cambridge, 2010.

\bibitem{Durand:2013}
F.~Durand.
\newblock Decidability of uniform recurrence of morphic sequences.
\newblock {\em Internat. J. Found. Comput. Sci.}, 24:123--146, 2013.

\bibitem{Durand&Host&Skau:1999}
F.~Durand, B.~Host, and C.~Skau.
\newblock Substitutive dynamical systems, bratteli diagrams and dimension
  groups.
\newblock {\em Ergodic Theory Dynam. Systems}, 19:953--993, 1999.

\bibitem{Dykstra&Rudolph:2010}
A.~Dykstra and D.~J. Rudolph.
\newblock Any two irrational rotations are nearly continuously {K}akutani
  equivalent.
\newblock {\em J. Anal. Math.}, 110:339--384, 2010.

\bibitem{Ferenczi:1997}
S.~Ferenczi.
\newblock Systems of finite rank.
\newblock {\em Colloq. Math.}, 73:35--65, 1997.

\bibitem{Ferenczi:2006}
S.~Ferenczi.
\newblock Substitution dynamical systems on infinite alphabets.
\newblock {\em Ann. Inst. Fourier (Grenoble)}, 56:2315--2343, 2006.
\newblock Num{\'e}ration, pavages, substitutions.

\bibitem{Ferenczi&Holton&Zamboni:2003}
S.~Ferenczi, C.~Holton, and L.~Q. Zamboni.
\newblock Structure of three-interval exchange transformations. {II}. {A}
  combinatorial description of the trajectories.
\newblock {\em J. Anal. Math.}, 89:239--276, 2003.

\bibitem{Ferenczi&Mauduit&Nogueira:1996}
S.~Ferenczi, C.~Mauduit, and A.~Nogueira.
\newblock Substitution dynamical systems: algebraic characterization of
  eigenvalues.
\newblock {\em Ann. Sci. \'Ecole Norm. Sup.}, 29:519--533, 1996.

\bibitem{Fogg:2002}
N.~P. Fogg.
\newblock {\em Substitutions in dynamics, arithmetics and combinatorics},
  volume 1794 of {\em Lecture Notes in Mathematics}.
\newblock Springer-Verlag, Berlin, 2002.
\newblock Edited by V. Berth{\'e}, S. Ferenczi, C. Mauduit and A. Siegel.

\bibitem{Giordano&Putnam&Skau:1995}
T.~Giordano, I.~Putnam, and C.~Skau.
\newblock Topological orbit equivalence and {$C^*$}-crossed products.
\newblock {\em J. Reine Angew. Math.}, 469:51--111, 1995.

\bibitem{Gjerde&Johansen:2002}
R.~Gjerde and 0.~Johansen.
\newblock Bratteli-vershik models for cantor minimal systems associated to
  interval exchange transformations.
\newblock {\em Math. Scand.}, 90:87--100, 2002.

\bibitem{Glasner&Weiss:1995}
E.~Glasner and B.~Weiss.
\newblock Weak orbit equivalence of cantor minimal systems.
\newblock {\em Internat. J. Math.}, 6:559--579, 1995.

\bibitem{Herman&Putnam&Skau:1992}
R.~H. Herman, I.~Putnam, and C.~F. Skau.
\newblock Ordered bratteli diagrams, dimension groups and topological dynamics.
\newblock {\em Internat. J. Math.}, 3:827--864, 1992.

\bibitem{Jullian:2013}
Y.~Jullian.
\newblock An algorithm to identify automorphisms which arise from self-induced
  interval exchange transformations.
\newblock {\em Math. Z.}, 274:33--55, 2013.

\bibitem{Kosek&Ormes&Rudolph:2008}
W.~Kosek, N.~Ormes, and D.~J. Rudolph.
\newblock Flow-orbit equivalence for minimal {C}antor systems.
\newblock {\em Ergodic Theory Dynam. Systems}, 28:481--500, 2008.

\bibitem{KryBog:1937}
N.~Kryloff and N.~Bogoliouboff.
\newblock La th\'eorie g\'en\'erale de la mesure dans son application \`a
  l'\'etude des syst\`emes dynamiques de la m\'ecanique non lin\'eaire.
\newblock {\em Ann. of Math. (2)}, 38:65--113, 1937.

\bibitem{Kurka:2003}
P.~Kurka.
\newblock {\em Topological and symbolic dynamics}, volume~11 of {\em Cours
  Sp\'ecialis\'es [Specialized Courses]}.
\newblock Soci\'et\'e Math\'ematique de France, Paris, 2003.

\bibitem{Lind&Marcus:1995}
D.~Lind and B.~Marcus.
\newblock {\em An Introduction to Symbolic Dynamics and Coding}.
\newblock Cambridge Univ. Press, 1995.

\bibitem{Mela&Petersen:2005}
Xavier M{\'e}la and Karl Petersen.
\newblock Dynamical properties of the {P}ascal adic transformation.
\newblock {\em Ergodic Theory Dynam. Systems}, 25(1):227--256, 2005.

\bibitem{Mosse:1992}
B.~Moss{\'e}.
\newblock Puissances de mots et reconnaissabilit\'e des points fixes d'une
  substitution.
\newblock {\em Theoret. Comput. Sci.}, 99:327--334, 1992.

\bibitem{Mosse:1996}
B.~Moss{\'e}.
\newblock Reconnaissabilit\'e des substitutions et complexit\'e des suites
  automatiques.
\newblock {\em Bull. Soc. Math. France}, 124:329--346, 1996.

\bibitem{Ornstein&Rudolph&Weiss:1982}
D.~S. Ornstein, D.~J. Rudolph, and B.~Weiss.
\newblock Equivalence of measure preserving transformations.
\newblock {\em Mem. Amer. Math. Soc.}, 37(262):xii+116, 1982.

\bibitem{Petersen:1983}
K.~Petersen.
\newblock {\em Ergodic theory}.
\newblock Cambridge University Press, 1983.

\bibitem{Queffelec:1987}
M.~Queff{\'e}lec.
\newblock {\em Substitution dynamical systems---spectral analysis}, volume 1294
  of {\em Lecture Notes in Mathematics}.
\newblock Springer-Verlag, Berlin, 1987.

\bibitem{Rauzy:1979}
G.~Rauzy.
\newblock \'{E}changes d'intervalles et transformations induites.
\newblock {\em Acta Arith.}, 34:315--328, 1979.

\bibitem{Rauzy:1982}
G.~Rauzy.
\newblock Nombres alg\'ebriques et substitutions.
\newblock {\em Bull. Soc. Math. France}, 110:147--178, 1982.

\bibitem{Roychowdhury&Rudolph:2009}
M.~R. Roychowdhury and D.~J. Rudolph.
\newblock Nearly continuous {K}akutani equivalence of adding machines.
\newblock {\em J. Mod. Dyn.}, 3:103--119, 2009.

\bibitem{Veech:1978}
W.~Veech.
\newblock Interval exchange transformations.
\newblock {\em J. Analyse Math.}, 33:222--272, 1978.

\bibitem{SWill}
S.~Williams.
\newblock Toeplitz minimal flows which are not uniquely ergodic.
\newblock {\em Z. Wahrsch. Verw. Gebiete}, 67:95--107, 1984.

\end{thebibliography}

\end{document}